\documentclass[11pt]{amsart}   	

\usepackage[margin=1in]{geometry}
\usepackage[T1]{fontenc}
\usepackage{lmodern}
\usepackage{amsmath,amssymb, amsfonts,amsthm,stackrel}
\usepackage[hide links]{hyperref}
\usepackage{mathtools}
\usepackage{hyperref}
\usepackage{tikz-cd}
\usepackage{array}
\tikzcdset{
	every diagram/.append style={
		every inside node/.style={text=.},
		draw=.
	}
}
\usetikzlibrary{decorations.markings}
\usepackage{makecell}
\setcellgapes{4pt}
\usepackage{multicol}
\usepackage[capitalise]{cleveref}
\usepackage[shortlabels]{enumitem}
\usepackage{stmaryrd}

\usepackage[colorinlistoftodos]{todonotes}

\newtheorem{theorem}{Theorem}[section]
\newtheorem{lemma}[theorem]{Lemma}
\newtheorem{corollary}[theorem]{Corollary}
\newtheorem{proposition}[theorem]{Proposition}
\newtheorem{definition-proposition}[theorem]{Definition-Proposition}
\newtheorem{conjecture}[theorem]{Conjecture}
\newtheorem{setup}[theorem]{Setup}

\theoremstyle{definition}
\newtheorem{definition}[theorem]{Definition}

\newtheorem{remark}[theorem]{Remark}

\DeclareMathOperator{\add}{\mathsf{add}}
\DeclareMathOperator{\Filt}{\mathsf{Filt}}
\DeclareMathOperator{\Gen}{\mathsf{Gen}}

\DeclareMathOperator{\Simp}{\mathsf{Simp}}
\DeclareMathOperator{\Cogen}{\mathsf{Cogen}}

\DeclareMathOperator{\ann}{\mathsf{ann}}
\DeclareMathOperator{\FiltGen}{\mathsf{T}}

\DeclareMathOperator{\FiltCogen}{\mathsf{F}}

\newcommand{\J}{\mathsf J}

\newcommand{\T}{\mathcal T}
\newcommand{\U}{\mathcal U}
\newcommand{\X}{\mathcal X}
\newcommand{\F}{\mathcal F}
\newcommand{\W}{\mathcal W}
\DeclareMathOperator{\ttop}{\mathsf{top}}
\DeclareMathOperator{\tors}{\mathsf{tors}}

\DeclareMathOperator{\mods}{\mathsf{mod}}

\DeclareMathOperator{\Hom}{\mathrm{Hom}}
\DeclareMathOperator{\Ext}{\mathrm{Ext}}

\DeclareMathOperator{\End}{\mathrm{End}}

\DeclareMathOperator{\rad}{\mathrm{rad}}

\renewcommand{\P}{\mathcal{P}}

\newcommand{\Pmod}{P}
\newcommand{\Pmods}{P_\mathrm{s}}
\newcommand{\Pmodns}{P_\mathrm{ns}}

\newcommand{\op}{\mathrm{op}}

\newcommand{\V}{\mathcal V}
\newcommand{\A}{\mathcal A}
\newcommand{\B}{\mathcal B}

\newcommand{\SC}{\mathcal S}
\newcommand{\C}{\mathcal C}

\DeclareMathOperator{\rk}{\mathsf{rk}}

\newcommand{\lperp}[1]{^{\perp} #1 }

\usepackage{xcolor}
\colorlet{brm}{blue} 

\usepackage{soul}

\title{Lifting $\tau$-exceptional sequences}
\author{Aslak B. Buan}
\address{Department of Mathematical Sciences, Norwegian University of Science and Technology (NTNU), 7491 Trondheim, NORWAY}
\email{aslak.buan@ntnu.no}
\author{Eric J. Hanson}
\address{Department of Mathematics, North Carolina State University, Raleigh, NC 27695, USA}
\email{ehansonmath@gmail.com}
\author{Bethany R. Marsh}
\address{School of Mathematics, University of Leeds, Leeds LS2 9JT, U.K.}
\email{B.R.Marsh@leeds.ac.uk}

\begin{document}
	\begin{abstract}
		We study the behaviour of $\tau$-exceptional sequences under \sloppy passage from a finite-dimensional algebra $\Lambda$ to a factor algebra $\Lambda/I$. 
		Assuming that $\Lambda$ is $\tau$-tilting finite and that $I\subseteq \operatorname{rad}\Lambda$, we construct, for every $r$ less than or equal to the number of simple $\Lambda$-modules, an injective map from the $\tau$-exceptional sequences of length $r$ over $\Lambda/I$ to those over $\Lambda$. The construction is recursive and is based on lifting indecomposable $\tau$-rigid modules and comparing their associated $\tau$-perpendicular categories. At each stage, we construct a right exact functor between the relevant $\tau$-perpendicular categories which, after identifying them with module categories, is naturally isomorphic to extension of scalars along a quotient $\Gamma\to\Gamma/J$, where $J\subseteq \operatorname{rad}\Gamma$. These functors send each lifted term to an object in the additive closure of the corresponding original term. Finally, we conjecture connections with known bijections for representations of quivers over local algebras and for Nakayama algebras.
	\end{abstract}
	\thanks{ABB was supported by 
		grant number FRINAT 301375 from the Norwegian Research Council. EJH was supported by an AMS-Simons travel grant. This work was supported by the Engineering and Physical Sciences Research Council [grant number EP/W007509/1] and by a CAS-alumni grant from the Centre for Advanced Study, Oslo.   }
	
	\maketitle
	
	\section{Introduction}

	The notion of $\tau$-exceptional sequences was introduced in \cite{bm} as a generalization of exceptional sequences over hereditary algebras, using ideas from $\tau$-tilting theory \cite{air} and $\tau$-tilting reduction \cite{jasso}.
	
	A $\tau$-exceptional sequence is a finite sequence of indecomposable objects in the category $\mods \Lambda$
	of finite-dimensional modules over an arbitrary fixed finite-dimensional algebra
	$\Lambda$, which induces a flag of wide subcategories of $\mods \Lambda$. Such sequences exist for every finite-dimensional algebra, and the length of the sequence is bounded by $\rk \Lambda$, the number of isomorphism classes of simple $\Lambda$-modules.
	
	Let $\tau$ be the Auslander-Reiten translate in $\mods \Lambda$, where $\Lambda$ is a finite-dimensional algebra. 
	A module $M$ is \emph{$\tau$-rigid} if $\Hom(M,\tau M) = 0$. We let
	$M^{\perp}$ denote the full subcategory of $\mods \Lambda$ generated by objects $\{X \mid \Hom(M,X)= 0 \}$, with ${^\perp M}$ defined dually. The \emph{$\tau$-perpendicular category} 
	$\J(M) = M^{\perp} \cap {^\perp \tau M}$ is, 
	by \cite{jasso}, equivalent to a module category $\mods \Lambda'$.  If $M$ is indecomposable, then $\rk \Lambda' = \rk \Lambda -1$. 
	
	A sequence 
	$(M_1, \dots, M_t)$ of indecomposable objects in $\mods \Lambda$ is a
	\emph{$\tau$-exceptional sequence} \cite{bm}
	if $M_t$ is $\tau$-rigid and the sequence
	$(M_1, \dots, M_{t-1})$ is a $\tau$-exceptional sequence in the $\tau$-perpendicular category $\J(M_t) = M_t^{\perp} \cap {^\perp \tau M_t}$. Note that the Auslander-Reiten translate is in general
	different in $\mods \Lambda$ and in $\J(M_t)$, so e.g., $M_{t-1}$ is not necessarily $\tau$-rigid in $\mods \Lambda$. 
	
	However, 
	we have that 
	each object $M_i$ in the sequence satisfies $\Ext^1(M_i,M_i) = 0$. This
	follows from the Auslander-Reiten formula (see \cite{ass06}) and the fact that 
	$\tau$-perpendicular categories are wide subcategories \cite[Thm. 4.12]{dirrt}. This explains the nomenclature. Note that we necessarily have that $t \le \rk \Lambda $. 
	
	An algebra is \emph{$\tau$-tilting finite} if there are finitely many isomorphism classes of indecomposable $\tau$-rigid modules, see \cite{dij}.
	We study $\tau$-exceptional sequences of factor algebras of $\tau$-tilting finite algebras, and obtain the following main result.
	
	\begin{theorem}\label{thm:main1}
		Let $\Lambda$ be a $\tau$-tilting finite algebra with $\rk \Lambda =n$, and let $I$ be an ideal contained in the radical of $\Lambda$. Then, for any positive integer $r$ with $1\leq r\leq n = \rk \Lambda$ there is an injective map
		$$
		\Delta_I \colon \left\{ \text{\parbox{4.5cm}{\centering $\tau$-exceptional sequences in $\mods \Lambda/I$ of length $r$}} \right\} 
		\rightarrow
		\left\{ \text{\parbox{4.5cm}{\centering $\tau$-exceptional sequences in $\mods \Lambda$ of length $r$}} \right\}
		$$
	\end{theorem}
	
	The map $\Delta_I$ is defined recursively. 
	An important ingredient is the construction of extension-of-scalars functors for the induced 
	$\tau$-perpendicular categories.
	
	For an ideal $I$ in $\Lambda$, we  
	let $F\colon \mods \Lambda \to \mods \Lambda /I$ denote the 
	functor $F = \frac{\Lambda}{I} \otimes_{\Lambda} -$.
	Let $\W$ be a $\tau$-perpendicular category of $\mods \Lambda$,
	and let $X$ be an object in $\W$ which is $\tau_{\W}$-rigid, that is $\tau$-rigid with respect to the Auslander-Reiten translate in $\W$. We then let $\J_{\W}(X)$ denote the corresponding 
	$\tau$-perpendicular category. This is also a $\tau$-perpendicular 
	category in $\mods \Lambda$, by \cite[Thm. 4.3]{bm2} in the $\tau$-tilting finite case, and by   \cite[Cor. 6.7]{bh} in the general case.
	
	Let $(C_1, \dots, C_t)$ be a $\tau$-exceptional sequence in $\mods \Lambda/I$,
	and let $\Delta_I(C_1, \dots, C_t) =(Z_1, \dots, Z_t)$. 
	Consider the wide subcategories $\W_i$ of $\mods \Lambda/I$ defined recursively by letting
	$\W_t = \mods \Lambda/I$ and for $1 \le s < t$ letting $\W_s = \J_{\W_{s+1}}(C_{s+1})$.
	Similarly, consider the wide subcategories $\V_i$ of $\mods \Lambda $ defined recursively by letting
	$\V_t = \mods \Lambda$ and for $1 \le s < t$ letting $\V_s = \J_{\V_{s+1}}(Z_{s+1})$.
	
	\begin{theorem}\label{thm:main2}
		With the notation above, we have $F(Z_t) \in \add C_t$. Moreover,
		there are functors $F_i \colon \V_i \to \W_i$, naturally isomorphic to 
		extension-of-scalars functors in the sense that for each $i<t$ there is a 
		diagram 
		
		$$
		\begin{tikzcd}
			\V_i \arrow[r, "\sim"] \arrow[d, "F_i"'] &  \mods \Gamma_i \arrow[d, "(\Gamma_i/I_i) \otimes_{\Gamma_i} -"] \\
			\W_i \arrow[r, "\sim"]                  & \mods (\Gamma_i/I_i)
		\end{tikzcd}
		$$
		which commutes up to natural isomorphism, 
		for some finite-dimensional algebra $\Gamma_i$ with an ideal $I_i$
		contained in the radical of $\Gamma_i$. 
		
		These functors satisfy $F_i(Z_i) \in \add C_i$.
	\end{theorem}	
	
	The paper is organized as follows. We give some background and summarize some known results in Section \ref{sec:background}.
	In Section \ref{sec:supp} we make a detour to establish a result needed concerning support $\tau$-tilting modules. In Section \ref{sec:ind} we deal with the lifting of an indecomposable $\tau$-rigid module, which is recursively applied in Section \ref{sec:rec}, where we give a proof of our main results.
	In the final two sections, we state some conjectures and give examples. 
	
	
	\section{Background}\label{sec:background}
	
	\subsection{General notions and torsion}
	
	Throughout, $k$ is a field and $\Lambda$ is a finite-dimensional
	$k$-algebra.  We write $\mods\Lambda$ for the category of
	finite-dimensional left $\Lambda$-modules. We consider modules up to isomorphism. We write
	$\rk\Lambda$ for the number of simple $\Lambda$-modules.  All subcategories are assumed to be full, additive and closed under isomorphisms. For a class
	$\X$ of modules, we denote by $\add\X$ the subcategory
	of direct summands of finite direct sums of objects in $\X$, by
	$\Gen\X$ the subcategory of factor modules of objects in
	$\add\X$, and by $\Cogen\X$ the subcategory of
	submodules of objects in $\add\X$.  We let
	$\Filt\X$ be the subcategory of modules admitting a finite
	filtration whose factors belong to $\X$.  When
	$\X=\{X\}$, for an object $X$, then we omit the braces.
	
	For a subcategory $\X\subseteq\mods\Lambda$, set
	$$
	{}^{\perp_\Lambda}\X
	=\{M\in\mods\Lambda\mid
	\operatorname{Hom}_\Lambda(M,X)=0
	\text{ for all }X\in\X\},
	$$
	and define $\X^{\perp_\Lambda}$ dually.  The ambient algebra
	will be omitted when it is clear. 
	Recall that a class $\T$ of objects in $\mods\Lambda$ is a \emph{torsion class} (respectively, \emph{torsion-free class}) if it is closed under quotients, direct sums and extensions (respectively, under submodules, direct sums and extensions).
	We denote by
	$\FiltGen_\Lambda(\X)$ the smallest torsion class containing
	$\X$, and by $\FiltCogen_\Lambda(\X)$ the smallest
	torsion-free class containing $\X$. 
	A \emph{wide subcategory} is a
	full subcategory closed under kernels, cokernels and extensions.  If
	$\W$ is wide, then $\Simp\ W$ denotes a set of
	representatives of its simple objects and
	$\rk\W=|\Simp\W|$.  We write
	$\ttop_\Lambda M=M/\rad_\Lambda M$, and $\rad\Lambda$ for the
	Jacobson radical of $\Lambda$.
	
	A torsion pair in $\mods\Lambda$ is a pair
	$(\T,\F)$ such that
	$\operatorname{Hom}_\Lambda(\T,\F)=0$ and every
	$X\in\mods\Lambda$ occurs in a canonical short exact sequence
	$$
	0\longrightarrow t_{\T}(X)\longrightarrow X
	\longrightarrow f_{\T}(X)\longrightarrow 0
	$$
	with $t_{\T}(X)\in\T$ and
	$f_{\T}(X)\in\F$.
	If $\T = \Gen U$ for some module $U$, we write $t_U$ and $f_U$ for
	$t_{\T}$ and $f_{\T}$.
	Note that every torsion class $\T$ is part of a torsion pair $(\T,\T^{\perp_{\Lambda}})$, while every torsion-free class $\F$ is part of a torsion pair $({}^{\perp_{\Lambda}}\F,\F)$. 

	A subcategory $\C$ is
	\emph{contravariantly finite} if every module has a right
	$\C$-approximation, and \emph{covariantly finite} if every module has
	a left $\C$-approximation; it is \emph{functorially finite} if it has
	both properties.
	
	Let $\T$ be a torsion class.  An object $P\in\T$ is
	\emph{Ext-projective in $\T$} if
	$\Ext^1_\Lambda(P,\T)=0$.  We denote the full
	subcategory of such objects by $\P(\T)$, and we let 
	$\Pmod(\T)$ be the basic direct sum of the indecomposable 
	Ext-projective objects in $\T$.
	An object
	$P\in\T$ is \emph{split projective in $\T$} if every
	epimorphism $X\twoheadrightarrow P$ with $X\in\T$ splits.
	When $\T$ is functorially finite, let
	$\Pmods(\T)$ be the basic direct sum of its indecomposable split
	projective objects, and let $\Pmodns(\T)$ be the basic direct sum
	of the indecomposable Ext-projective objects which are not split
	projective. Thus the basic Ext-projective generator is
	$$
	\Pmod(\T)=\Pmods(\T)\amalg\Pmodns(\T),
	\qquad
	\P(\T)=\add P(\T).
	$$
	
	\subsection{$\tau$-tilting}
	
	Let $\tau=\tau_\Lambda$ denote the Auslander--Reiten translate in 
	$\mods \Lambda$. We recall some key notions from \cite{air}.
	A module $M$ is \emph{$\tau$-rigid} if
	$\operatorname{Hom}_\Lambda(M,\tau M)=0$.  Writing $|M|$ for the
	number of nonisomorphic indecomposable direct summands of a basic
	module $M$, we call $M$ a \emph{support $\tau$-tilting module} if it is
	$\tau$-rigid and there is a projective module $P$ such that
	$$
	\operatorname{Hom}_\Lambda(P,M)=0,
	\qquad |M|+|P|=\rk\Lambda.
	$$
	The Adachi--Iyama--Reiten correspondence \cite[Thm. 2.7]{air} identifies basic support
	$\tau$-tilting modules with functorially finite torsion classes via
	$M\mapsto\Gen M$, with inverse
	$\T\mapsto \Pmod(\T)$. 
	
	A basic module $M=\amalg_{i=1}^r M_i$, with the $M_i$
	indecomposable, is \emph{gen-minimal} if
	$$
	M_i\notin
	\Gen\Bigl(\amalg_{j\ne i}M_j\Bigr)
	\qquad\text{for every }i.
	$$
	The assignments
	$$
	M\longmapsto\Gen M,
	\qquad
	\T\longmapsto\Pmods(\T)
	$$
	give mutually inverse bijections between isomorphism classes of basic
	gen-minimal $\tau$-rigid modules and functorially finite torsion
	classes. This variation of the Adachi--Iyama--Reiten correspondence is explicit in \cite[Thm. 1.1]{bhm2}.

	The algebra $\Lambda$ is
	\emph{$\tau$-tilting finite} if it has only finitely many isomorphism classes
	of indecomposable $\tau$-rigid modules; equivalently, every torsion
	class in $\mods\Lambda$ is functorially finite, by~\cite[Thm.\ 1.2]{dij}.
	
	\subsection{Factor algebras}
	
	For an ideal $I \subseteq  \Lambda$, recall that
	$F\colon \mods \Lambda \to \mods \Lambda /I$ denotes the 
	functor $F = \frac{\Lambda}{I} \otimes_{\Lambda} -$.

	\begin{lemma}\label{lem:Gen-preserved}
		\begin{itemize}
			\item[(a)] Let $M$ be $\tau$-rigid in $\mods \Lambda$.
			Then we have $$F(\Gen M) = \Gen F(M) = \Gen M \cap \mods \Lambda/I.$$
			\item[(b)] If $\Lambda$ is $\tau$-tilting finite, then $F(\T)= \T \cap \mods \Lambda/I$ for any torsion class $\T$ in $\mods \Lambda$.
		\end{itemize}
	\end{lemma}
	
	\begin{proof}
		The inclusion $F(\Gen M) \subseteq \Gen F(M)$ follows from $F$
		being right exact. For the opposite inclusion,
		let $X \in \Gen F(M) \subseteq \Gen M \cap \mods \Lambda/I$.
		So $X = F(X)$ with $X \in \Gen M$, that is $F(\Gen M) \supseteq \Gen F(M)$.
		The second equality in (a) now follows from \cite[Prop. 5.6(b)]{dirrt}.
		
		Part (b) now follows from the fact that all torsion classes are finitely generated, since $\Lambda$ is $\tau$-tilting finite \cite[Thm. 3.8]{dij}.
	\end{proof}
	
	We will use the following notion and result from \cite{dirrt}.
	
	\begin{definition}
		For a subcategory $\SC$ of $\mods \Lambda/I$, let 
		$\phi^+(\SC) = \{X \in \mods \Lambda \mid F(X) \in \SC \}$.
		\end{definition}
		
		\begin{proposition}\cite[Prop. 5.7]{dirrt}\label{prop:dirrt-facts}
			\begin{itemize}
				\item[(a)] If $(\T, \F)$ is a torsion pair in $\mods \Lambda$, then
				$(\T \cap \mods \Lambda/I, \F \cap \mods \Lambda/I)$ is a  torsion pair in $\mods \Lambda/I$.
				\item[(b)] If $\T'$ is a torsion class in $\mods \Lambda/I$, then $\phi^+(\T') \cap \mods \Lambda/I = \T'$.
			\end{itemize}
		\end{proposition}

		\section{Support \texorpdfstring{$\tau$}{tau}-tilting modules}\label{sec:supp}
		
		In this section, we give a characterization of support $\tau$-tilting modules, which is particularly useful when comparing 
		the $\tau$-tilting theories and torsion theories of $\mods \Lambda$ and $\mods \Lambda/I$.
		It is a variation of results in \cite[Sect. 1]{air} and 
		\cite[Sect. 3]{ms}. 
		
		We first make the following observation.
		\begin{lemma}
			\label{lem:addTenough}
			Suppose that $T$ is a $\Lambda$-module.
			Let $g:\Lambda\rightarrow T_0$ be a left $\add T$-approximation. Then $g$ is a left $\Gen T$-approximation.
		\end{lemma}
		\begin{proof}
			Let $X\in \Gen T$. Then there is an epimorphism $\pi:\widetilde{T}\rightarrow X$ where $\widetilde{T}\in \add T$.
			Let $h:\Lambda\rightarrow X$. Then, since $\Lambda$ is projective, there is a map $u:\Lambda\rightarrow \widetilde{T}$ such that $h=\pi u$. Since $g$ is a left $\add T$-approximation, there is a map $v:T_0\rightarrow \widetilde{T}$ such that $u=vg$, so $h=\pi u=\pi v g$ factors through $g$ as required.
		\end{proof}
		
		For a $\Lambda$-module $M$, we denote the annihilator of $M$ by
		$$\ann M=\{a\in \Lambda\,:\,\text{$am=0$ for all $m\in M$}\}.$$
		We also recall that a $\Lambda$-module $T$ is said to be \emph{tilting}
		if $\Ext^1(T,T)=0$, the projective dimension of $T$ is at most $1$, and
		there is an exact sequence:
		$$
		\begin{tikzcd}
			0 \arrow[r] &\Lambda \arrow[r] & T_0 \arrow[r] & T_1 \arrow[r] & 0
		\end{tikzcd}
		$$
		of $\Lambda$-modules where $T_0$ and $T_1$ lie in $\add T$.
		
		\begin{lemma}
			\label{lem:sequencegivessupport}
			Suppose that there is a right exact sequence:
			$$
			\begin{tikzcd}
				\Lambda \arrow[r,"g"] & T_0 \arrow[r] & T_1 \arrow[r] & 0
			\end{tikzcd}
			$$
			where $T_0, T_1 \in \add T$ for a $\tau$-rigid module $T$, and $g$ is a left $\add T$-approximation of $\Lambda$. Let $\T=\Gen T$. Then:
			\begin{itemize}
				\item[(a)] There is an induced short exact sequence
				$$\begin{tikzcd}
					0 \arrow[r] & \Lambda/\ann \T \arrow[r,"\widetilde{g}"] & T_0 \arrow[r] & T_1 \arrow[r] & 0
				\end{tikzcd}
				$$
				\item[(b)] The $\Lambda$-module $T$ is a tilting module over $\Lambda/\ann \T$.
				\item[(c)] We have $\P(\T)=\add T$.
				\item[(d)] The module $T$ is a support $\tau$-tilting $\Lambda$-module.
			\end{itemize}
		\end{lemma}
		\begin{proof}
			Note that, by Lemma~\ref{lem:addTenough}, the map $g$ is a left $\Gen T$-approximation.
			
			We first show that $\ker g=\ann \T$. Let $a\in \ann \T$, and set $g(1)=t\in T_0$. Then $g(a)=g(a\cdot 1)=ag(1)=at=0$, so $a\in \ker g$ and we have $\ann \T\subseteq \ker g$.
			
			Conversely, let $a\in \ker g$. Let $U\in \T$ and let $v\in U$.
			Define a homomorphism $\alpha:\Lambda\rightarrow U$ by setting
			$\alpha(x)=xv$ for $x\in \Lambda$.
			Since $g$ is a left $\T$-approximation of $\Lambda$, there is a map $u:T_0\rightarrow U$ satisfying $\alpha=ug$.
			$$\begin{tikzcd}
				\Lambda \arrow[r,"g"] \arrow[d,"\alpha"] & T_0 \arrow[dl,dashed,"u"] \\
				U
			\end{tikzcd}$$
			Then, since $a\in \ker g$, we have
			$$av=\alpha(a)=(ug)(a)=u(g(a))=0,$$
			so $a\in \ann \T$, giving $\ker g\subseteq \ann \T$ and proving the claim. Part (a) follows.
			
			Since $T$ is $\tau$-rigid, it follows from~\cite[Prop. 1.2(b)]{air} that $T\in \P(\Gen T)=\P(\T)$.
			Since $\T$ is a functorially finite torsion class, it follows from~\cite[Prop. 1.1(e)]{air} that $P(\T)$ is a tilting module over $\Lambda/\ann \T$.
			Hence, $\Ext^1_{\Lambda/\ann \T}(T,T)=0$ and
			the projective dimension of $T$ over $\Lambda/\ann T$ is at most $1$. Together with (a), this shows that $T$ is a tilting module over $\Lambda/\ann \T$, giving part (b).
			
			Since $T\in \P(\T)$ and both $T$ and $P(\T)$ are tilting modules over $\Lambda/\ann \T$, we must have $\add T =\add P(\T)$ (noting that the additive subcategory generated by a module in $\mods \Lambda/\ann \T$ is the same as the additive subcategory generated as a module in $\mods\Lambda$). Hence $\add T=\P(\T)$, giving part (c).
			
			Finally, by~\cite[Thm. 2.7]{air}, the module $T$ is a support $\tau$-tilting $\Lambda$-module, giving part (d).
		\end{proof}
		
		\begin{corollary}
			\label{cor:supporttau}
			Let $T$ be a $\Lambda$-module. Then $T$ is a support $\tau$-tilting module if and only if $T$ is $\tau$-rigid and there is an exact sequence
			\begin{equation}
				\label{eq:rightsequence}
				\begin{tikzcd}
					\Lambda \arrow[r,"g"] & T_0 \arrow[r] & T_1 \arrow[r] & 0,
				\end{tikzcd}
			\end{equation}
			where $T_0$ and $T_1$ lie in $\add T$ and $g$ is a left $\add T$-approximation of $\Lambda$.
		\end{corollary}
		\begin{proof}
			First, suppose that $T$ is a support $\tau$-tilting module. Then, by definition, $T$ is $\tau$-rigid. Let
			$g:\Lambda\rightarrow T_0$ be a minimal left $\T=\Gen T$-approximation of $\Lambda$. Completing this to an exact
			sequence, we obtain
			$$
			\begin{tikzcd}
				\Lambda \arrow[r,"g"] & T_0 \arrow[r] & T_1 \arrow[r] & 0.
			\end{tikzcd}
			$$
			By~\cite[Lemma 3.7]{ms}, we have $\add(T_0\amalg T_1)=\P(\T)$.
			Since $T$ is a support $\tau$-tilting module, we have that $T=P(\T)$ by~\cite[Thm.\ 2.7]{air}, and hence $T_0,T_1\in \add T$ as required.
			
			Conversely, assume that $T$ is $\tau$-rigid and that there is an exact sequence of the form~\eqref{eq:rightsequence} where
			$T_0$ and $T_1$ lie in $\add T$ and $g$ is a left $\add T$-approximation of $\Lambda$. Since $T_0$ and $T_1$ lie in $\add T$, we have that $g$ is an $\add(T_0\coprod T_1)$-approximation of $\Lambda$ so, by Lemma~\ref{lem:sequencegivessupport}, it follows that $T$ is a support $\tau$-tilting module.
		\end{proof}
		
		A consequence of Corollary~\ref{cor:supporttau} is the following, which is attributed to~\cite{dirrt} in~\cite[Prop.\ 5.1]{iy}.
		
		\begin{proposition} \cite{dirrt}, \cite[Prop. 5.1]{iy}
			\label{prop:Fsupporttau}
			Let $T$ be a support $\tau$-tilting $\Lambda$-module. Then $F(T)$ is a support $\tau$-tilting $\Lambda/I$-module.
		\end{proposition}
		\begin{proof}
			Let $T$ be a support $\tau$-tilting $\Lambda$-module and let $\T=\Gen T$. By Corollary~\ref{cor:supporttau}, there is a right exact sequence
			of $\Lambda$-modules:
			$$
			\begin{tikzcd}
				\Lambda \arrow[r,"g"] & T_0 \arrow[r] & T_1 \arrow[r] & 0,
			\end{tikzcd}
			$$
			where $T_0$ and $T_1$ lie in $\add T$ and $g$ is a left $\add T$-approximation of $\Lambda$. Since $F$ is right exact, we obtain a right exact sequence of $\Lambda/I$-modules:
			\begin{equation}
				\label{eq:FTexact}
				\begin{tikzcd}
					\Lambda/I \arrow[r,"F(g)"] & F(T_0) \arrow[r] & F(T_1) \arrow[r] & 0.
				\end{tikzcd}
			\end{equation}

			We claim that $F(g)$ is a left $F(\add T)$-approximation. Let $S\in \add T$ and let $\pi_{\Lambda}:\Lambda\rightarrow \Lambda/I$ and $\pi_S:S\rightarrow F(S)$ be the natural quotient maps. Suppose we are given a map $u:\Lambda/I\rightarrow F(S)$. Since $\Lambda$ is projective, the map $u\pi_{\Lambda}$ factors through $\pi_S$ giving $u\pi_{\Lambda}=\pi_Sv$ for some map $v:\Lambda\rightarrow S$. Then $u=F(v)$. Since $g$ is a left $\add T$-approximation, there is a map $w:T_0\rightarrow S$ such that $v=wg$. Then $u=F(v)=F(w)F(g)$ and $u$ factors through $F(g)$ as required.
			$$\begin{tikzcd}
				\Lambda \arrow[dr,"v",dashed] \arrow[dd,"\pi_{\Lambda}"] \arrow[rr,"g"] && T_0 \arrow[rr] \arrow[dl,"w",dashed] && T_1 \arrow[rr] && 0 \\
				& S \arrow[dd,"\pi_S",pos=0.25] & \\
				\Lambda/I \arrow[rr,"F(g)",pos=0.25] \arrow[dr,"u"] && F(T_0) \arrow[rr] \arrow[dl,"F(w)",dashed] && F(T_1) \arrow[rr] && 0 \\
				& F(S)
			\end{tikzcd}
			$$
			Since $T$ is a $\tau$-rigid $\Lambda$-module, the 
			$\Lambda/I$-module $F(T)$ is $\tau_{\Lambda/I}$-rigid  by~\cite[Prop.\ 5.6(a)]{dirrt}. Noting the exact sequence~\eqref{eq:FTexact}, it follows from Corollary~\ref{cor:supporttau} that $F(T)$ is a support
			$\tau_{\Lambda/I}$-tilting $\Lambda/I$-module as required.
		\end{proof}
		
		\section{Lifting a $\tau$-rigid indecomposable module}\label{sec:ind}
		Our proof of Theorem \ref{thm:main1} is recursive. In this section, we consider the lifting of the rightmost object in a $\tau$-exceptional sequence, which is by definition a $\tau$-rigid indecomposable object $C$
		in $\mods \Lambda/I$. We assume in this section that $\Lambda$ is $\tau$-tilting finite, and that $I \subseteq \rad \Lambda$.
		Note that by \cite[Thm. 5.12]{dirrt}, the factor algebra $\Lambda/I$ is also $\tau$-tilting finite, for any ideal $I$.
		For a given factor algebra $\Lambda/I$, with the 
		extension-of-scalars functor $F = \Lambda/I \otimes_{\Lambda} -$, there may in general be several lifts of $C$ to indecomposable $\tau$-rigid $\Lambda$-modules. We need to choose such a lift $Z$ with the property that there is a functor $F_Z:\J(Z)\rightarrow \J_{\Lambda/I}(C)$ with similar properties to those of $F$. The lift that we define in this section is shown to have such properties in the next section.
		
		\begin{setup}
			\label{setupCZ}
			Let $C\in\mods \Lambda/I$ be an indecomposable $\tau_{\Lambda/I}$-rigid module. 
			Let $\W = \Filt \J_{\Lambda/I}(C)$ and let $Z= \Pmods({}^{\perp}\W)$.
		\end{setup}
		
		We now state and prove a series of lemmas, giving properties of the relationship between $Z$ and $C$ and the related torsion and torsion-free categories, and the $\tau$-perpendicular categories.
		
		\begin{lemma}\label{lem:known}
			In Setup~\ref{setupCZ}, the following hold:
			\begin{itemize}
				\item[(a)] $\W$ is a wide subcategory of $\mods \Lambda$;
				\item[(b)] $\Gen_{\Lambda/I}C={}^{\perp_{\Lambda/I}}\J_{\Lambda/I}(C)$;
				\item[(c)] $C=\Pmods({}^{\perp_{\Lambda/I}}\J_{\Lambda/I}(C))$.
			\end{itemize}
		\end{lemma}
		
		\begin{proof}
			For (a), note that since 
			$\J_{\Lambda/I}(C)$ is a wide subcategory of $\mods \Lambda/I$, we have that 
			$\Simp \J_{\Lambda/I}(C)$ is a semibrick in $\mods \Lambda/I$, and hence also in 
			$\mods \Lambda$. It follows that
			$$\W = \Filt \J_{\Lambda/I}(C) =  \Filt \Simp \J_{\Lambda/I}(C)$$ is a wide subcategory of $\mods \Lambda$.
			
			(b) follows from \cite[Thm.\ 2.14]{bhm}, using that $C$ is indecomposable and hence gen-minimal.
			
			(c) follows from (b), using \cite[\S2.3]{air} (cf.~\cite[Thm.\ 1.6(b)]{bhm}).
		\end{proof}

		\begin{lemma}
			\label{lem:CZinitial}
			In Setup~\ref{setupCZ}, the following hold:
			\begin{itemize}
				\item[(a)] $Z$ is gen-minimal and $\tau$-rigid;
				\item[(b)] $\J(Z)=\W=\Filt \J_{\Lambda/I}(C)$ and so $\J(Z)$ and $\J_{\Lambda/I}(C)$ have the same rank;
				\item[(c)] $Z$ is indecomposable;
				\item[(d)] $Z=\Pmods(\Gen Z)$;
				\item[(e)] $\Gen Z={}^{\perp}\J(Z)$;
				\item[(f)] $Z$ is projective in $\mods \Lambda$ if and only if $C$ is 
				projective in $\mods \Lambda/I$. In this case $F(Z)=C$. 
			\end{itemize}
		\end{lemma}
		\begin{proof}
			For (a), note that $Z$ is gen-minimal by~\cite[\S 2.3]{air} (see~\cite[Thm.\ 1.6]{bhm}) and $\tau$-rigid by~\cite[Thm.\ 2.7]{air}.
			
			We then have:
			\begin{align*}
				{}^{\perp}\J(Z) &= {}^\perp{}\J(\Pmods({}^{\perp}W)) \\
				&= \Gen(\Pmods({}^{\perp}W)) \quad \text{by~\cite[Thm.\ 2.14(a)]{bhm}} \\
				&= {}^{\perp}\W.
			\end{align*}
			By~\cite[VI.1.2]{ass06}, we have the torsion pairs:
			$$({}^{\perp}\J(Z),({}^{\perp}\J(Z))^{\perp})=({}^{\perp}\J(Z),\FiltCogen (\J(Z))),$$
			and
			$$({}^{\perp}\W,({}^{\perp}\W)^{\perp})=({}^{\perp}\W,\FiltCogen (\W)).$$
			Since ${}^{\perp}\J(Z)={}^{\perp}\W$, we have $\FiltCogen (\J(Z))=\FiltCogen (\W)$.
			Now, recall that in the $\tau$-tilting finite case, the map $\FiltCogen$ is a bijection from wide subcategories to torsion-free classes (\cite{ms}, see also \cite[Thm. 2.10]{bhm}).
			It follows that $\J(Z)=\W$. Furthermore, we have $\Simp \J(Z)=\Simp \W=\Simp \J_{\Lambda/I}(C)$ by~\cite[Thm.\ 1.2]{ringel}, so $\W$ has rank $\rk \Lambda-1$, as this is the rank of $\J_{\Lambda/I}(C)$ by~\cite[Thm.\ 3.8]{jasso}. This proves (b).
			
			Using again~\cite[Thm.\ 3.8]{jasso}, we have that $Z$ is indecomposable, giving (c).
			
			Then part (d) follows from part (a) and~\cite[Sect.\ 2.3]{air} (see~\cite[Thm.\ 1.6(b)]{bhm}). Part
			(e) follows from \cite[Thm.\ 2.14]{bhm}, using that $Z$ is gen-minimal by (a).
			
			For (f), first note that since $I\subseteq \rad \Lambda$, we have $\Simp \mods \Lambda = \Simp \mods \Lambda/I$.
			
			If $Z$ is projective, then $F(Z)$ is indecomposable projective in $\mods \Lambda/I$, with $\ttop Z = \ttop F(Z)$. 
			We then have $\Simp \J(Z)= \{S \mid S \in \Simp{\mods \Lambda}\} \setminus \{\ttop Z\}$.
			Since  $\Simp \J_{\Lambda/I}(C) = \Simp \J(Z)$, we have
			$\J_{\Lambda/I} (C) 
			=  \Filt_{\Lambda/I} (\{S \mid S \in \Simp{\mods \Lambda/I}\} \setminus \{\ttop F(Z) \}) = F(Z)^{\perp_{\Lambda/I}} = \J_{\Lambda/I}(F(Z))$. 
			By \cite[Prop. 6.15]{bh}, we then have $F(Z)=C$ and $C$ is hence projective.  
			
			Assume $C$ is projective in $\mods \Lambda/I$ and let $Z'$ be the (indecomposable) projective in $\mods \Lambda$ such that $F(Z') = C$. Then, arguing as above, the wide subcategories $\J(Z')$ and 
			$\J_{\Lambda/I}(F(Z'))$ have the same simples. Hence, we have $\J(Z) = \J(Z')$, which by \cite[Prop. 6.15]{bh} implies that $Z=Z'$.
		\end{proof}

		Note that for any subcategory $\X$ of $\mods \Lambda/I$, viewed as a full subcategory of $\mods \Lambda$, we have 
		\begin{equation}\label{eq-c}
			\X^{\perp} \cap \mods \Lambda/I = \X^{\perp_{\Lambda/I}} \text{     and     }
			{^{\perp}{\X}} \cap \mods \Lambda/I =	{^{\perp_{\Lambda/I}}{\X}}.
		\end{equation}
		
		\begin{lemma}
			\label{lem:TJC}
			In Setup~\ref{setupCZ}, we have
			$$\FiltGen (\J_{\Lambda/I}(C))\cap \mods \Lambda/I=\FiltGen_{\Lambda/I}(\J_{\Lambda/I}(C)).$$
		\end{lemma}
		
		\begin{proof}
			Let $\X =\J_{\Lambda/I}(C)$.
			We claim that $$\FiltGen (\X) \cap \mods \Lambda/I ={\lperp{(\X^{\perp})}} \cap \mods \Lambda/I
			= {^{\perp_{\Lambda/I}} {(\X^{\perp})}}
			\subseteq {^{\perp_{\Lambda/I}} {(\X^{\perp_{\Lambda/I}})}} =\FiltGen_{\Lambda/I}(\X).$$
			
			The inclusion follows from $\X^{\perp_{\Lambda/I}}
			\subseteq \X^{\perp}$.
			We have that $\FiltGen_{\Lambda/I}(\X)$ is the smallest torsion class in $\mods \Lambda/I$ containing $\X$. But $\FiltGen (\X) \cap \mods \Lambda/I$ is a torsion class 
			by \cref{prop:dirrt-facts} and contains $\X$. So we
			must have equality as claimed.
		\end{proof}

		\begin{lemma}
			\label{lem:FGenZ}
			In Setup~\ref{setupCZ}, we have
			$F(\Gen Z) = \Gen_{\Lambda/I}C$ and $Z^\perp \cap \mods \Lambda/I  = C^{\perp_{\Lambda/I}}$.
		\end{lemma}

		\begin{proof}
			We have $\J(Z) = \Filt \J_{\Lambda/I}(C)$ by \cref{lem:CZinitial}(b), and hence
			
			$${^\perp{\J(Z)}}= {^\perp{\J_{\Lambda/I}}(C)}.$$
			We intersect with $\mods \Lambda/I$ and obtain
			\begin{equation}\label{eq-a} 
				^\perp{\J(Z)} \cap \mods \Lambda/I = {^\perp{\J_{\Lambda/I}(C)}} \cap \mods \Lambda/I = 
				{^{\perp_{\Lambda/I}}\J_{\Lambda/I}(C)} = \Gen_{\Lambda/I}C,
			\end{equation}
			where the last equation is by \cref{lem:known}(b).
			
			Furthermore, by \cref{lem:CZinitial}(e), we have $\Gen Z = {^\perp{\J(Z)}}$, which by \cref{lem:Gen-preserved} implies 
			\begin{equation}\label{eq-b}
				F(\Gen Z) = F({^\perp{\J(Z)}}) = {^\perp{\J(Z)}} \cap \mods \Lambda/I.
			\end{equation}
			
			The first statement now follows from combining \cref{eq-a} and \cref{eq-b}. 
			Since $\Gen_{\Lambda/I} C = F(\Gen Z) = \Gen Z \cap \mods \Lambda/I$ by \cref{lem:Gen-preserved}, we then have that the 
			second statement follows from \cref{prop:dirrt-facts}(a).
		\end{proof}
		
		\begin{lemma}\label{lem:perptau}
			If $X$ is an indecomposable non-projective $\tau$-rigid $\Lambda$-module, then ${^\perp{(\tau X)}}= \FiltGen (\J(X))$.
		\end{lemma}	
		
		\begin{proof}
			Using that $\Lambda$ is $\tau$-tilting finite, we have the following bijections of subcategories of $\mods \Lambda$ (\cite[Lem. 4.3]{bh})
			$$
			\begin{tikzcd}
				\{\text{torsion classes}\} \arrow[rrr, shift left=0.6ex, "\T \mapsto \J(\Pmodns(\T)) \cap P^{\perp}"] & & & \{\text{wide subcategories}\} \arrow[lll, shift left=0.6ex, "\FiltGen (\W) \mapsfrom \W"]
			\end{tikzcd}
			$$
			where $P$ is the maximal (basic) projective module satisfying $\Hom(P, \T)=0$. 
			
			The torsion class ${^\perp{\tau X}}$ is sincere \cite[Thm. 2.10]{air}. Moreover we have that $\Pmodns({^\perp{\tau X}}) =X$, from \cite[Lem. 1.7]{bhm}. Hence, the above bijection gives  $^\perp{\tau X} = \FiltGen (\J(X))$.   
		\end{proof}

		\begin{lemma}
			\label{lem:Fperptau}
			In Setup~\ref{setupCZ}, we have $F({}^{\perp}\tau Z)={}^{\perp_{\Lambda/I}}\tau_{\Lambda/I}C$.
		\end{lemma}

		\begin{proof}
			By \cref{lem:CZinitial}(f) we have that $Z$ is projective in $\mods \Lambda$ if and only if  	
			$C$ is projective in $\mods \Lambda/I$. In this case, we have ${}^{\perp}\tau Z = \mods \Lambda$ and 
			${}^{\perp_{\Lambda/I}}\tau_{\Lambda/I}C = \mods \Lambda/I$, so the result clearly holds.  
			
			Now assume that $Z$ is not projective.	
			By \cref{lem:perptau}, we have that  ${^\perp{\tau Z}}= \FiltGen (\J(Z))$ and also 
			${^{\perp_{\Lambda/I}}{(\tau_{\Lambda/I} C)}}= \FiltGen_{\Lambda/I} (\J_{\Lambda/I}(C))$.

			Hence, we have:
			\begin{align*}
				F({}^{\perp} \tau Z) &= F(\FiltGen (\J(Z))) \\
				&= F(\FiltGen( \Filt \J_{\Lambda/I}(C))) \text{\quad (by the definition of $\W=\J(Z)$)} \\
				&= F(\FiltGen (\J_{\Lambda/I}(C))) \text{\quad (since $\X^{\perp} = (\Filt \X)^{\perp}$ and hence 
					$\FiltGen (\X) = \FiltGen (\Filt \X)$ for any subcategory $\X$)} \\
				&= \FiltGen (\J_{\Lambda/I}(C))\cap \mods \Lambda/I \text{\quad (by 
					Lemma~\ref{lem:Gen-preserved}(b))} \\
				&= \FiltGen_{\Lambda/I} (\J_{\Lambda/I}(C)) \text{\quad (by Lemma~\ref{lem:TJC})} \\
				&= {}^{\perp_{\Lambda/I}}\tau_{\Lambda/I}C,
			\end{align*}
			which gives the result.
		\end{proof}	
		
		The final result of this section gives a crucial property for our main results, proved in the next section. 
		
		\begin{proposition}
			\label{prop:FZaddC}
			In Setup~\ref{setupCZ}, we have 
			\begin{itemize}
				\item[(a)]$\add F(\Pmod(\lperp{\tau Z})) = \add 
				\Pmod(^{\perp_{\Lambda/I}}{\tau_{\Lambda/I} C})$ and
				\item[(b)] $F(Z)=C^r$ for some $r\geq 1$.
			\end{itemize}
		\end{proposition}
		\begin{proof}
			Note that $P({}^{\perp}\tau Z)$ is a support $\tau$-tilting module, by~\cite[Thm.\ 2.7]{air}. Hence, by Proposition~\ref{prop:Fsupporttau}, we have that $F(P({}^{\perp}\tau Z))$ is a support $\tau_{\Lambda/I}$-tilting module.
			We have, using Lemma~\ref{lem:Gen-preserved}, that
			\begin{align*}
				\Gen F(P({}^{\perp}\tau Z)) &= F(\Gen P({}^{\perp}\tau Z)) \\
				&=F({}^{\perp}\tau Z) \\
				&={}^{\perp_{\Lambda/I}}\tau_{\Lambda/I}C \text{\quad (by Lemma~\ref{lem:Fperptau})} \\
				&= \Gen P({}^{\perp_{\Lambda/I}}\tau_{\Lambda/I}C) \text{\quad (by~\cite[Thm. 2.7]{air})}.
			\end{align*}    
			Since $F(P({}^{\perp}\tau Z))$ and $P({}^{\perp_{\Lambda/I}}\tau_{\Lambda/I}C)$ are both support $\tau_{\Lambda/I}$-tilting in $\mods \Lambda/I$, we have $\add F(P({}^{\perp}\tau Z))=\add P({}^{\perp_{\Lambda/I}}\tau_{\Lambda/I}C)$ by~\cite[Thm.\ 2.7]{air}. This establishes part (a).
			
			Similarly, using Lemma~\ref{lem:Gen-preserved} again, we have:
			\begin{align*}
				\Gen F(P(\Gen Z)) &=F(\Gen P(\Gen Z)) \\
				&= F(\Gen Z) \\
				&= \Gen_{\Lambda/I} C \text{\quad (by Lemma~\ref{lem:FGenZ})} \\
				&= \Gen P(\Gen_{\Lambda/I} C) \text{\quad (by~\cite[Thm.\ 2.7]{air})}.
			\end{align*}
			Since $F(P(\Gen Z))$ is a support $\tau$-tilting module (by~\cite{dirrt}) and $P(\Gen_{\Lambda/I} C)$ is also a support $\tau$-tilting module (by~\cite[Thm.\ 2.7]{air}), we have that $\add F(P(\Gen Z))=\add P(\Gen_{\Lambda/I} C)$.
			
			Note that $F(Z)\in F(P({}^{\perp}\tau Z)) \cap F(P(\Gen Z))$, so $F(Z)\in \add P({}^{\perp_{\Lambda/I}}\tau_{\Lambda/I}C)\cap \add P(\Gen_{\Lambda/I}C)$.
			So $F(Z)\in \add C$ by~\cite[Lem. 1.7]{bhm} (noting that $F(Z)$ lies in the Bongartz and co-Bongartz completions of $C$).
			
			Note then that $\Gen_{\Lambda/I}C=F(\Gen Z)=\Gen F(Z)$ by Lemma~\ref{lem:FGenZ} and Lemma~\ref{lem:Gen-preserved}(a).
			Since $C\not=0$, we must have $F(Z)\not=0$, so $F(Z)=C^r$ for some $r\geq 1$, giving part (b).
		\end{proof}
		
		\section{Reduction}\label{sec:rec}

		We first recall the following theorem of Jasso.
		
		\begin{theorem}[{\cite[Theorem~1.4]{jasso}}]\label{thm:jasso-module-category}
			Let $M$ be a $\tau$-rigid module in $\mods\Lambda$. Then
			$
			\J(M)\simeq \mods\Gamma
			$
			for some finite-dimensional algebra $\Gamma$.
		\end{theorem}
		
		Throughout this section, let $Z$ and $C$ be as in Setup \ref{setupCZ}. 
		
		\subsection{Torsion-free classes and torsion-free reflectors}
		
		We recall some categorical notions. 
		
		\begin{definition}\label{def:pre-abelian}
			Let $\A$ be an additive category.
			\begin{enumerate}[(a)]
				\item The category $\A$ is \emph{pre-abelian} if every morphism
				in $\A$ has a kernel and a cokernel.
				
				\item A full subcategory $\F$ of a category $\A$ is
				\emph{reflective} if the inclusion
				$$
				\iota\colon \F\hookrightarrow\A
				$$
				admits a left adjoint. This left adjoint is called the
				\emph{reflector} onto $\F$.
				
				\item Let $\A$ and $\B$ be additive categories with
				cokernels. An additive functor
				$$
				G\colon\A\longrightarrow\B
				$$
				is \emph{right exact} if it preserves cokernels.
			\end{enumerate}
		\end{definition}
		
		Let $(\T,\F)$ be a torsion pair in an abelian category
		$\A$. For every $X\in\A$, we write
		$$
		0\longrightarrow t(X)\longrightarrow X
		\xrightarrow{\eta_X} f(X)\longrightarrow 0
		$$
		for its canonical torsion sequence, where $t(X)\in\T$ and
		$f(X)\in\F$.
		
		The following description of the categorical structure of $\F$
		is well-known; see \cite[Sec.~5.4]{bondal-vdb}. We include a proof for convenience.
		
		\begin{proposition}\label{prop:intrinsic-preabelian}
			Let $\A$ be an abelian category and let
			$(\T,\F)$ be a torsion pair in $\A$. Then the
			following hold.
			\begin{enumerate}[(a)]
				\item The functor
				$$
				f\colon\A\longrightarrow\F
				$$
				is left adjoint to the inclusion
				$\iota\colon\F\hookrightarrow\A$. In particular, we have that
				$\F$ is a reflective subcategory of $\A$ and $f$ is
				its reflector.
				
				\item The category $\F$ is pre-abelian. More precisely, let
				$a\colon X\to Y$ be a morphism in $\F$, and let
				$$
				K=\ker_{\A}(a),\qquad
				q\colon Y\longrightarrow N=\operatorname{coker}_{\A}(a).
				$$
				Then
				$$
				\ker_{\F}(a)=K,
				\qquad
				\operatorname{coker}_{\F}(a)=f(N),
				$$
				where the cokernel morphism in $\F$ is the composition
				$
				Y\xrightarrow{q}N\xrightarrow{\eta_N}f(N).
				$
			\end{enumerate}
		\end{proposition}
		
		\begin{proof}
			Let $X\in\A$ and $C\in\F$. Since
			$\Hom_{\A}(\T,\F)=0$, every morphism
			$X\to C$ vanishes on $t(X)$ and therefore factors uniquely through
			the canonical epimorphism
			$\eta_X\colon X\longrightarrow f(X).$
			
			This gives a natural isomorphism
			$$
			\Hom_{\F}(f(X),C)
			\simeq
			\Hom_{\A}(X,C),
			$$
			proving (a).
			
			For (b), since a torsion-free class is closed under subobjects,
			$K=\ker_{\A}(a)$ belongs to $\F$. Hence the kernel
			of $a$ in $\A$ has the universal property of a kernel also
			in $\F$.
			
			Now let $F\in\F$ and let $h\colon Y\to C$ satisfy $ha=0$.
			Since $q$ is the cokernel of $a$ in $\A$, there is a unique
			morphism $\widetilde h\colon N\to C$ such that
			$
			h=\widetilde hq.
			$
			By part~(a), the map $\widetilde h$ factors uniquely through $\eta_N$, so
			there is a unique morphism $\overline h\colon f(N)\to C$ such that
			$h=\overline h\,\eta_Nq.$
			Thus $\eta_Nq$ is the cokernel of $a$ in $\F$, as claimed.
		\end{proof}
		
		We next record the right-exactness consequence which will be used
		below.
		
		\begin{proposition}\label{prop:torsionfree-reflector}
			Let $\A$ be an abelian category and let $(\T,\F)$
			be a torsion pair in $\A$, with reflector
			$$
			f\colon \A \longrightarrow \F.
			$$
			Equip $\F$ with the pre-abelian structure described in the
			previous proposition. Then the following hold.
			\begin{enumerate}[(a)]
				\item The functor
				$$
				f\colon \A \longrightarrow \F
				$$
				is right exact.
				
				\item Let $\B$ be an abelian category, let
				$$
				E\colon \B \longrightarrow \A
				$$
				be right exact, and let $\W\subseteq \F$ be a wide
				subcategory of $\A$. If $fE(X)\in\W$ for every
				$X\in\B$, then the composition
				$$
				fE\colon \B \longrightarrow \W
				$$
				is right exact.
			\end{enumerate}
		\end{proposition}
		
		\begin{proof}
			By the Proposition \ref{prop:intrinsic-preabelian}(a), the functor
			$$
			f\colon \A \longrightarrow \F
			$$
			is left adjoint to the inclusion $\F\hookrightarrow\A$.
			Hence $f$ preserves colimits, and in particular cokernels. This proves
			(a).
			
			For (b), let $q\colon Y\to N$ be the cokernel of a morphism
			$a\colon X\to Y$ in $\B$. Since $E$ is right exact,
			$E(q)$ is the cokernel of $E(a)$ in $\A$. By part~(a),
			$$
			fE(q)\colon fE(Y)\longrightarrow fE(N)
			$$
			is the cokernel of $fE(a)$ in $\F$. Since all the objects
			involved belong to the full subcategory $\W$, the same
			morphism has the universal property of a cokernel in $\W$.
			Thus
			$$
			fE\colon \B \longrightarrow \W
			$$
			is right exact.
		\end{proof}
		
		\begin{remark}
			A pre-abelian category $\A$ is \emph{quasi-abelian} if kernels
			are stable under pushout and cokernels are stable under pullback.
			In fact, the pre-abelian category $\F$ above is quasi-abelian
			(called \emph{almost abelian} in \cite{rump}); this follows from
			\cite[Theorem~2]{rump}. Moreover, the exact structure on $\F$
			induced from the abelian category $\A$ is intrinsic: its exact
			sequences are precisely the short exact sequences in $\A$
			whose terms belong to $\F$. See
			\cite[Sec.~5.4]{bondal-vdb}, where the induced exact
			structure on either member of a torsion pair is described intrinsically
			in terms of its kernels and cokernels. We will not need these stronger
			properties below.
		\end{remark}
		
		\subsection{The induced functor between the perpendicular categories}
		
		We now apply \cref{prop:torsionfree-reflector} to our setting.
		
		\begin{proposition}\label{prop:FZ-defined}
			The composite $f_CF$ restricts to a well-defined right exact functor
			$$
			F_Z:=f_CF|_{\J(Z)}:
			\J(Z)\longrightarrow\J_{\Lambda/I}(C).
			$$
			Moreover, for every $X\in{}^{\perp}\tau Z$, the canonical
			epimorphism $p_X:X\to f_Z(X)$ induces a natural isomorphism
			$$
			\rho_X:f_CF(X)\xrightarrow{\ \sim\ }F_Zf_Z(X).
			$$
			Consequently, the diagram
			\begin{equation}\label{dia:comm}
				\begin{tikzcd}[column sep=large,row sep=large]
					{}^{\perp}\tau Z
					\arrow[r,"f_Z"]
					\arrow[d,"F"']
					&
					\J(Z)
					\arrow[d,"F_Z"]
					\\
					{}^{\perp_{\Lambda/I}}\tau_{\Lambda/I}C
					\arrow[r,"f_C"']
					&
					\J_{\Lambda/I}(C)
				\end{tikzcd}
			\end{equation}
			commutes up to natural isomorphism.
		\end{proposition}
		
		\begin{proof}
			Let $X\in\J(Z)$. Since
			$
			\J(Z)\subseteq{}^{\perp}\tau Z,
			$
			Lemma \ref{lem:Fperptau} gives
			$
			F(X)\in{}^{\perp_{\Lambda/I}}\tau_{\Lambda/I}C.
			$
			Since ${}^{\perp_{\Lambda/I}}\tau_{\Lambda/I}C$ is a torsion
			class and therefore closed under factor modules, we have 
			$
			f_C(F(X))\in{}^{\perp_{\Lambda/I}}\tau_{\Lambda/I}C.
			$
			By definition,
			$
			f_C(F(X))\in C^{\perp_{\Lambda/I}}.
			$
			It follows that
			$$
			f_C(F(X))\in
			C^{\perp_{\Lambda/I}}
			\cap{}^{\perp_{\Lambda/I}}\tau_{\Lambda/I}C
			=
			\J_{\Lambda/I}(C),
			$$
			so $F_Z$ is well defined.
			
			Since $\J(Z)$ is a wide subcategory of $\mods\Lambda$, the cokernels of morphisms in $\J(Z)$ are computed in $\mods\Lambda$. Since $F$ is
			right exact, its restriction
			$
			F|_{\J(Z)}:
			\J(Z)\longrightarrow\mods\Lambda/I
			$
			is right exact. The category $\J_{\Lambda/I}(C)$ is a wide subcategory of
			$\mods\Lambda/I$ contained in $C^{\perp_{\Lambda/I}}$. Hence
			Proposition~\ref{prop:torsionfree-reflector}(c), applied to the torsion pair
			$(\Gen_{\Lambda/I}C,C^{\perp_{\Lambda/I}})$, shows that $F_Z$ is right
			exact.
			
			Now let $X\in{}^{\perp}\tau Z$ and consider the canonical
			exact sequence 
			$$
			0\longrightarrow t_Z(X)\xrightarrow{i_X} X
			\xrightarrow{p_X}f_Z(X)\longrightarrow0
			$$
			with respect to the torsion pair
			$(\Gen Z,Z^{\perp})$.
			Since ${}^{\perp}\tau Z$ is a torsion class, it is closed
			under factor modules. Hence
			$
			f_Z(X)\in
			Z^{\perp}\cap{}^{\perp}\tau Z
			=
			\J(Z).
			$
			By Lemma \ref{lem:FGenZ}, we have
			$F(\Gen Z)=\Gen_{\Lambda/I}C.$
			Since $t_Z(X)\in\Gen Z$, this gives
			$F(t_Z(X))\in\Gen_{\Lambda/I}C,$
			and therefore
			$f_CF(t_Z(X))=0.$
			
			By the right exactness of $F$ and \cref{prop:torsionfree-reflector}(b), the composite
			$$
			f_C F : \mods\Lambda
			\longrightarrow C^{\perp_{\Lambda/I}}
			$$
			preserves cokernels. Hence $f_C F(p_X)$ is a cokernel of
			$f_C F(i_X)$. Since $f_C F(t_Z(X))=0$, the morphism
			$f_C F(i_X)$ is the zero morphism
			$0 \longrightarrow f_C F(X).$
			The identity morphism of $f_C F(X)$ is also a cokernel of this zero
			morphism. It follows from the universal property of cokernels that
			$$
			\rho_X := f_C F(p_X)\colon f_C F(X)
			\longrightarrow f_C F(f_Z(X))
			$$
			is an isomorphism.

			Since $f_Z(X)\in \J(Z)$, we have by definition
			$$
			f_CF(f_Z(X))=F_Zf_Z(X).
			$$
			Moreover, the canonical torsion sequences are functorial in $X$.
			Therefore, the morphisms $\rho_X$ form a natural isomorphism
			$$
			\rho\colon
			\left.f_CF\right|_{{}^{\perp}\tau Z}
			\xrightarrow{\sim}
			F_Zf_Z.
			$$
			
			This is precisely the natural isomorphism expressing the commutativity of
			the diagram (\ref{dia:comm}).
		\end{proof}
		
		We next consider the following property of the torsion-free functor.
		
		\begin{lemma}\label{lem:fU-projectives}
			Let $U$ be a $\tau$-rigid $\Lambda$-module. The torsion-free functor $f_U$
			induces a full and dense functor
			$$
			f_U:\P({}^{\perp}\tau U)
			\longrightarrow
			\P(\J(U)).
			$$
			Its kernel consists precisely of the morphisms which factor through an
			object of $\Gen U$.
		\end{lemma}
		
		\begin{proof}
			We first identify the kernel of
			$$
			f_U:\mods\Lambda\longrightarrow U^{\perp}.
			$$
			Let $h:X\to X'$ be a morphism. Functoriality of the canonical torsion
			sequences gives a commutative diagram
			$$
			\begin{tikzcd}
				0 \arrow[r] & t_U(X) \arrow[r] \arrow[d] & X \arrow[r] \arrow[d,"h"]
				& f_U(X) \arrow[r] \arrow[d,"f_U(h)"] & 0\\
				0 \arrow[r] & t_U(X') \arrow[r] & X' \arrow[r]
				& f_U(X') \arrow[r] & 0.
			\end{tikzcd}
			$$
			If $f_U(h)=0$, then $h$ factors through $t_U(X')\in\Gen U$.
			Conversely, every morphism factoring through an object of
			$\Gen U$ is annihilated by $f_U$. Thus the kernel consists exactly
			of the morphisms which factor through $\Gen U$.
			
			By \cite[Lemma~4.9]{bm}, the restriction of $f_U$ induces a
			dense functor
			$
			\P({}^{\perp}\tau U)
			\longrightarrow
			\P(\J(U)).
			$
			It remains to prove fullness. Let $Q,Q'\in
			\P({}^{\perp}\tau U)$, and let
			$
			h':f_U(Q)\longrightarrow f_U(Q')
			$
			be a morphism. Consider the canonical exact sequence
			$$
			0\longrightarrow t_U(Q')\longrightarrow Q'
			\longrightarrow f_U(Q')\longrightarrow0.
			$$
			The category ${}^{\perp}\tau U$ is generated by the
			Bongartz completion $B_U$, and
			$
			t_U(Q')\in\Gen U\subseteq\Gen  B_U.
			$
			Since $Q$ is Ext-projective in ${}^{\perp}\tau U$, we have
			$\Ext^1_\Lambda(Q,t_U(Q'))=0.$
			Hence $h'$ lifts to a morphism $h:Q\to Q'$ satisfying
			$$
			f_U(h)=h'.
			$$
			Thus the restriction of $f_U$ is full.
		\end{proof}
		
		The following crucial property of the the induced functor $F_Z$ will be the key to the recursive step in our main result.
		
		\begin{proposition}\label{prop:FZ-projectives}
			The restriction
			$$
			F_Z|_{\P(\J(Z))}:
			\P(\J(Z))
			\longrightarrow
			\P(\J_{\Lambda/I}(C))
			$$
			is full and dense.
		\end{proposition}
		
		\begin{proof}
			Let $Q\in\P({}^{\perp}\tau Z)$. By
			\cref{prop:FZaddC}(a),
			$$
			\add F\bigl(P({}^{\perp}\tau Z)\bigr)
			=
			\add P({}^{\perp_{\Lambda/I}}\tau_{\Lambda/I}C).
			$$
			Since $Q$ belongs to
			$\add P({}^{\perp}\tau Z)$, it follows that
			$$
			F(Q)\in
			\P({}^{\perp_{\Lambda/I}}\tau_{\Lambda/I}C).
			$$
			Lemma~\ref{lem:fU-projectives}, applied in $\mods\Lambda/I$ to $C$, gives
			$$
			f_CF(Q)\in\P(\J_{\Lambda/I}(C)).
			$$
			By Proposition~\ref{prop:FZ-defined}, there is a natural isomorphism
			$$
			f_CF(Q)\xrightarrow{\ \sim\ }F_Zf_Z(Q).
			$$
			Since Lemma~\ref{lem:fU-projectives}, applied to $Z$, shows that every
			projective object of $\J(Z)$ is isomorphic to $f_Z(Q)$ for some
			$Q\in\P({}^{\perp}\tau Z)$, the functor $F_Z$ sends
			projective objects of $\J(Z)$ to projective objects of
			$\J_{\Lambda/I}(C)$.
			
			We next prove fullness. Let $P=f_Z(Q)$, $P'=f_Z(Q')$
			with $Q,Q'\in\P({}^{\perp}\tau Z)$, and let
			$w:F_Z(P)\longrightarrow F_Z(P')$
			be a morphism.
			
			Note that, for any $X\in {}^{\perp}\tau Z$, we have
			$F_Zf_Z(X)=f_cFf_Z(X)$ by definition, so there is a canonical epimorphism 
			$u_X:f_Z(X)\longrightarrow F_Zf_Z(X)=f_cFf_Z(X)$ which is natural in 
			$X$.
			Composing this with the canonical epimorphism $p_X:X\rightarrow f_Z(X)$
			gives an epimorphism $u_Xp_X:X\rightarrow F_Zf_Z(X)$.
			Since
			$f_Z(Q)$ is projective in $\J(Z)$ and $u_{Q'}$ is an epimorphism
			in $\J(Z)$, there is a morphism
			$
			v:f_Z(Q)\longrightarrow f_Z(Q')
			$
			such that
			$
			u_{Q'}v=wu_Q.
			$
			By Lemma~\ref{lem:fU-projectives}, the functor
			$$
			f_Z:\P({}^{\perp}\tau Z)
			\longrightarrow\P(\J(Z))
			$$
			is full. Hence there is a morphism $v':Q\to Q'$ satisfying
			$
			f_Z(v')=v.
			$
			Naturality of the maps $u_X$ gives a commutative square
			
			$$
			\begin{tikzcd}
				f_Z(Q) \arrow[rr,"u_Q"] \arrow[d,"v"] && F_Zf_ZQ \arrow[d,"F_Zv=F_Zf_Zv'"] \\
				f_Z(Q') \arrow[rr,"u_{Q'}"] && F_Zf_ZQ'
			\end{tikzcd}
			$$
			
			Using that $u_{Q'}v = w u_Q$, we obtain $wu_Q =F_Zvu_Q$, and
			since $u_Q$ is an epimorphism, we conclude that 
			$
			F_Zv=w.
			$
			Thus the restriction of $F_Z$ to projectives is full.
			
			It remains to prove density. First, \cref{prop:FZaddC}(a), together with the density
			of
			$$
			f_C:\P({}^{\perp_{\Lambda/I}}\tau_{\Lambda/I}C)
			\longrightarrow\P(\J_{\Lambda/I}(C)),
			$$
			shows that every projective object of $\J_{\Lambda/I}(C)$ is a direct
			summand of an object in the image of
			$F_Z|_{\P(\J(Z))}$. Thus this restriction is dense up to direct
			summands.
			
			Let
			$P_1,\ldots,P_n$
			be representatives of the isomorphism classes of indecomposable
			projective objects of $\J(Z)$. Note that if we have an epimorphism $P_i \to X$  in $\J(Z)$ with $X \neq 0$,
			then $X$ must be indecomposable. This is because $P_i$ has a simple top as an object in the finite length module category $\J(Z)$.

			By Lemma \ref{lem:CZinitial}(b), the wide category
			$\J_{\Lambda/I}(C)$ also has rank $n$. The natural transformation
			constructed in Proposition~\ref{prop:FZ-defined} gives, for each $i$, an
			epimorphism
			$P_i\longrightarrow F_Z(P_i)$
			in $\J(Z)$; here we use
			$\J_{\Lambda/I}(C)\subseteq\J(Z)$
			because $\J(Z)=\Filt\J_{\Lambda/I}(C)$. 
			Consequently, each $F_Z(P_i)$ is either zero or indecomposable.
			
			Every indecomposable projective object of $\J_{\Lambda/I}(C)$ occurs as a
			direct summand of some finite direct sum of objects $F_Z(P_i)$. Since the
			nonzero objects $F_Z(P_i)$ are indecomposable, every indecomposable
			projective object of the target is isomorphic to one of them. We therefore
			obtain a surjection
			$$
			\{P_i\mid F_Z(P_i)\neq0\}
			\longrightarrow
			\{\text{isomorphism classes of indecomposable projectives in }
			\J_{\Lambda/I}(C)\}.
			$$
			The target has cardinality $n$, while the source has cardinality at most
			$n$. Hence every $F_Z(P_i)$ is nonzero, the objects $F_Z(P_i)$ are pairwise
			nonisomorphic, and they exhaust the indecomposable projective objects of
			$\J_{\Lambda/I}(C)$. Thus the restriction of $F_Z$ is dense.
		\end{proof}
		
		\subsection{Proof of the main results}
		
		We first give a general statement for a functor between abelian categories. Most of this is classical, but we 
		include a proof for completeness.
		
		\begin{lemma}\label{lem:projective-generator-criterion}
			Let $\A$ and $\B$ be $k$-linear abelian categories which
			are equivalent to module categories of finite-dimensional $k$-algebras.
			Assume that
			$
			\rk\A=\rk\B.
			$
			Let
			$
			G:\A\longrightarrow\B
			$
			be a right exact $k$-linear functor whose restriction
			$$
			G|_{\P(\A)}:
			\P(\A)\longrightarrow\P(\B)
			$$
			is full and dense. Let $P$ be a basic projective generator of $\A$
			and let
			$\Gamma=\End_{\A}(P)^{\op}.$
			Then the following hold.
			
			\begin{itemize}
				\item[(a)]The object $P'=G(P)$ is a basic projective generator of $\B$.
				
				\item[(b)]
				The functor $G$ induces a surjective $k$-algebra homomorphism
				$$
				\pi:\Gamma\longrightarrow
				\Gamma':=\End_{\B}(P')^{\op}.
				$$
				If $J=\ker\pi$, then
				$$
				\Gamma'\simeq\Gamma/J
				\qquad\text{and}\qquad
				J\subseteq\rad\Gamma.
				$$
				
				\item[(c)]
				Under the equivalences induced by the projective generators $P$ and $P'$,
				the functor $G$ is naturally isomorphic to extension of scalars along
				$\pi$. More precisely, there is a diagram commuting up to natural
				isomorphism
				$$
				\begin{tikzcd}[column sep=large,row sep=large]
					\A
					\arrow[rr,"{\Hom_{\A}(P,-)}","\sim"']
					\arrow[d,"G"']
					& & \mods\Gamma
					\arrow[d,"{(\Gamma/J)\otimes_\Gamma-}"]
					\\
					\B
					\arrow[rr,"{\Hom_{\B}(P',-)}","\sim"']
					& &\mods(\Gamma/J).
				\end{tikzcd}
				$$
			\end{itemize}
		\end{lemma}
		
		\begin{proof}
			Let $n=\rk\A=\rk\B$ and write
			$P=P_1\amalg\cdots\amalg P_n,$
			where the $P_i$ are pairwise nonisomorphic indecomposable projective
			objects. Since the restriction of $G$ is full, it induces a surjective
			homomorphism
			$$
			\End_{\A}(P_i)\longrightarrow
			\End_{\B}(G(P_i)).
			$$
			The algebra $\End_{\A}(P_i)$ is local. Hence, if $G(P_i)\neq0$,
			then $\End_{\B}(G(P_i))$ is local, and therefore $G(P_i)$ is
			indecomposable.
			
			Let $Q$ be an indecomposable projective object of $\B$. Since the
			restriction of $G$ is dense, there is a projective object $R$ of
			$\A$ with $G(R)\simeq Q$. Decomposing $R$ into indecomposable
			projectives shows that $Q\simeq G(P_i)$ for some $i$. Thus the assignment
			$P_i\longmapsto G(P_i)$
			induces a surjection from the set of those $P_i$ for which $G(P_i)\neq0$
			onto the set of isomorphism classes of indecomposable projective objects
			of $\B$. The latter set has cardinality $n$, while the former has
			cardinality at most $n$. It follows that every $G(P_i)$ is nonzero and that
			the objects $G(P_i)$ are pairwise nonisomorphic. Therefore
			$$
			P'=G(P)=G(P_1)\amalg\cdots\amalg G(P_n)
			$$
			is a basic projective generator of $\B$.
			
			The functor $G$ induces a homomorphism
			$$
			\End_{\A}(P)\longrightarrow\End_{\B}(P'),
			\qquad h\longmapsto G(h),
			$$
			which is surjective because the restriction of $G$ is full. Passing to
			opposite algebras gives a surjective homomorphism
			$$
			\pi:\Gamma=\End_{\A}(P)^{\op}
			\longrightarrow
			\Gamma'=\End_{\B}(P')^{\op}.
			$$
			Let $J=\ker\pi$. Then $\Gamma'\simeq\Gamma/J$.
			
			It remains to prove that $J\subseteq\rad\Gamma$. The simple $\Gamma'$-modules
			are exactly the simple
			$\Gamma$-modules annihilated by $J$. Since
			$\rk\Gamma'=\rk\B=\rk\A=\rk\Gamma,$
			the simple $\Gamma'$-modules give all isomorphism classes of simple
			$\Gamma$-modules. Hence $J$ annihilates every simple $\Gamma$-module, and
			therefore
			$$
			J\subseteq
			\bigcap_{S\in\Simp\Gamma}\ann_\Gamma S
			=\rad\Gamma.
			$$
			
			Consider the equivalences
			$$
			H=\Hom_{\A}(P,-):\A\xrightarrow{\sim}\mods\Gamma
			$$
			and
			$$
			H'=\Hom_{\B}(P',-):\B\xrightarrow{\sim}\mods\Gamma'.
			$$
			For every projective object $Q$ of $\A$, define
			$$
			\theta_Q:\Gamma'\otimes_\Gamma\Hom_{\A}(P,Q)
			\longrightarrow
			\Hom_{\B}(P',G(Q))
			$$
			by
			$$
			\theta_Q(\gamma'\otimes h)=\gamma'\cdot G(h).
			$$
			This is well defined because the right $\Gamma$-module structure on
			$\Gamma'$ is induced by $\pi$, and $\pi$ itself is induced by $G$. The
			maps $\theta_Q$ are natural in $Q$. For $Q=P$, the map $\theta_P$ is the
			canonical isomorphism
			$$
			\Gamma'\otimes_\Gamma\Gamma\xrightarrow{\sim}\Gamma'.
			$$
			It is therefore an isomorphism for every finite direct sum of copies of
			$P$. Since every projective object of $\A$ is a direct summand of
			such a finite direct sum, naturality implies that $\theta_Q$ is an
			isomorphism for every $Q\in\P(\A)$. Thus
			$$
			\Gamma'\otimes_\Gamma H(-)
			\xrightarrow{\sim}
			H'G(-)
			$$
			on $\P(\A)$.
			
			Both functors
			$
			\Gamma'\otimes_\Gamma H(-)$ and 
			$H'G(-)$
			are right exact. A right exact functor from an abelian category with
			enough projectives is determined, up to natural isomorphism, by its
			restriction to the projective objects. Indeed, for any $X\in\A$,
			choose a projective presentation
			$$
			Q_1\longrightarrow Q_0\longrightarrow X\longrightarrow0.
			$$
			Applying the two right exact functors and using the natural isomorphism on
			$Q_1$ and $Q_0$ gives an induced natural isomorphism on their cokernels.
			Consequently,
			$
			\Gamma'\otimes_\Gamma H(-)
			\xrightarrow{\sim}
			H'G(-)
			$
			on all of $\A$. Since $\Gamma'\simeq\Gamma/J$, this gives the
			stated diagram.
		\end{proof}
		
		Applying \cref{lem:projective-generator-criterion} to our setting,
		we obtain the following.
		
		\begin{theorem}\label{thm:reduction-functor}
			Let $Z$ and $C$ be as in Setup \ref{setupCZ}. Choose an equivalence
			$\J(Z)\simeq\mods\Gamma,
			$
			where $\Gamma$ is basic. Then the following hold.
			
			\begin{itemize}
				\item[(a)]
				There is a right exact functor
				$$
				F_Z:\J(Z)\longrightarrow\J_{\Lambda/I}(C).
				$$
				
				\item[(b)]
				There is an ideal $I_Z\subseteq\Gamma$ such that
				$
				\J_{\Lambda/I}(C)\simeq\mods(\Gamma/I_Z),
				$
				and there is a diagram commuting up to natural isomorphism
				$$
				\begin{tikzcd}[column sep=large,row sep=large]
					\J(Z)
					\arrow[r,"\sim"]
					\arrow[d,"F_Z"']
					& \mods\Gamma
					\arrow[d,"{(\Gamma/I_Z)\otimes_\Gamma-}"]
					\\
					\J_{\Lambda/I}(C)
					\arrow[r,"\sim"']
					& \mods(\Gamma/I_Z).
				\end{tikzcd}
				$$
				
				\item[(c)]
				The ideal $I_Z$ is contained in $\rad\Gamma$.
			\end{itemize}
		\end{theorem}
		
		\begin{proof}
			Define
			$$
			F_Z=f_CF|_{\J(Z)}.
			$$
			By Proposition~\ref{prop:FZ-defined}, this is a well-defined right exact
			functor. By Proposition~\ref{prop:FZ-projectives}, its restriction to the
			projective objects is full and dense. By Lemma \ref{lem:CZinitial}(b), the categories
			$\J(Z)$ and $\J_{\Lambda/I}(C)$ have the same rank, and by
			Theorem~\ref{thm:jasso-module-category} they are equivalent to module
			categories of finite-dimensional algebras. Applying
			Lemma~\ref{lem:projective-generator-criterion} with
			$$
			\A=\J(Z),
			\qquad
			\B=\J_{\Lambda/I}(C),
			\qquad
			G=F_Z,
			$$
			gives a surjective algebra homomorphism
			$$
			\Gamma\longrightarrow\Gamma/I_Z
			$$
			with $I_Z\subseteq\rad\Gamma$, together with the required equivalence and
			commutative diagram.
		\end{proof}
		
		We can now prove the first main theorem.
		
		\begin{proof}[Proof of \cref{thm:main1}]
			We argue by
			induction on $r$.
			For $r=1$, a $\tau$-exceptional sequence in $\mods\Lambda/I$ is an
			indecomposable $\tau_{\Lambda/I}$-rigid module $C$. Let
			$$
			\Delta_I(C)=Z=P_s({}^{\perp}\J_{\Lambda/I}(C)).
			$$
			By Lemma \ref{lem:CZinitial}(c), we have that $Z$ is an indecomposable $\tau$-rigid module. If
			$\widetilde C$ is another indecomposable $\tau_{\Lambda/I}$-rigid module
			with the same lift $Z$, then \cref{prop:FZaddC}(b) gives
			$
			F(Z)\simeq C^a\simeq\widetilde C^{\,b}
			$
			for positive integers $a$ and $b$. By Krull--Schmidt, we get $C\simeq\widetilde C$.
			Thus the map is injective for $r=1$.
			
			Assume that $r>1$ and that the result holds for sequences of length
			$r-1$. Let
			$$
			(C_1,\ldots,C_r)
			$$
			be a $\tau$-exceptional sequence in $\mods\Lambda/I$, and let
			$Z_r=P_s({}^{\perp}\J_{\Lambda/I}(C_r)).$
			By Theorem~\ref{thm:reduction-functor}, there is a finite-dimensional
			algebra $\Gamma$ and an ideal $I_{Z_r}\subseteq\rad\Gamma$ such that
			$$
			\J(Z_r)\simeq\mods\Gamma,
			\qquad
			\J_{\Lambda/I}(C_r)\simeq\mods(\Gamma/I_{Z_r}).
			$$
			The algebra $\Gamma$ is $\tau$-tilting finite by Jasso's $\tau$-tilting
			reduction theorem \cite[Theorem~3.15]{jasso}. Hence the induction
			hypothesis, applied to $\Gamma$ and $I_{Z_r}$, gives an injective map
			$$
			\Delta':
			\left\{
			\begin{array}{c}
				\tau\text{-exceptional sequences in }\J_{\Lambda/I}(C_r)\\
				\text{of length }r-1
			\end{array}
			\right\}
			\longrightarrow
			\left\{
			\begin{array}{c}
				\tau\text{-exceptional sequences in }\J(Z_r)\\
				\text{of length }r-1
			\end{array}
			\right\}.
			$$
			Since $(C_1,\ldots,C_{r-1})$ is a $\tau$-exceptional sequence in
			$\J_{\Lambda/I}(C_r)$, write
			$$
			(Z_1,\ldots,Z_{r-1})
			=
			\Delta'(C_1,\ldots,C_{r-1})
			$$
			and define
			$$
			\Delta_I(C_1,\ldots,C_r)
			=(Z_1,\ldots,Z_{r-1},Z_r).
			$$
			By definition, this is a $\tau$-exceptional sequence in $\mods\Lambda$.
			
			To prove injectivity, suppose that
			$$
			\Delta_I(C_1,\ldots,C_r)
			=
			\Delta_I(\widetilde C_1,\ldots,\widetilde C_r).
			$$
			Then $Z_r=\widetilde Z_r$, and the argument for $r=1$ gives
			$C_r\simeq\widetilde C_r$. The same reduction category and the same map
			$\Delta'$ therefore occur for the first $r-1$ terms, and
			$$
			\Delta'(C_1,\ldots,C_{r-1})
			=
			\Delta'(\widetilde C_1,\ldots,\widetilde C_{r-1}).
			$$
			The induction hypothesis implies
			$$
			(C_1,\ldots,C_{r-1})
			=
			(\widetilde C_1,\ldots,\widetilde C_{r-1}),
			$$
			which proves the injectivity of $\Delta_I$.
		\end{proof}
		
		The second main theorem is also a direct consequence of the recursive
		construction.
		
		\begin{proof}[Proof of Theorem \ref{thm:main2}]
			Let
			$(C_1,\ldots,C_t)$
			be a $\tau$-exceptional sequence in $\mods\Lambda/I$, and let
			$$
			\Delta_I(C_1,\ldots,C_t)=(Z_1,\ldots,Z_t).
			$$
			Let $\V_i$ and $\W_i$ be the wide subcategories defined in the introduction.
			For the final term, set
			$$
			F_t=F:\V_t=\mods\Lambda\longrightarrow \W_t=\mods\Lambda/I.
			$$
			\cref{prop:FZaddC}(b) gives
			$$
			F_t(Z_t)\in\add C_t.
			$$
			
			Now fix $i<t$. At the corresponding stage of the recursive construction,
			$\V_i$ and $\W_i$ are the perpendicular categories obtained from
			$\V_{i+1}$ and $\W_{i+1}$ by reducing with respect to $Z_{i+1}$ and
			$C_{i+1}$, respectively. Theorem~\ref{thm:reduction-functor}, applied in
			the module categories equivalent to $\V_{i+1}$ and $\W_{i+1}$, gives a
			right exact functor
			$$
			F_i:\V_i\longrightarrow \W_i
			$$
			which, under suitable equivalences
			$$
			V_i\simeq\mods\Gamma_i,
			\qquad
			W_i\simeq\mods(\Gamma_i/I_i),
			$$
			is naturally isomorphic to
			$(\Gamma_i/I_i)\otimes_{\Gamma_i}-$
			for an ideal $I_i\subseteq\rad\Gamma_i$. At this stage $Z_i$ is, by
			construction, the chosen lift of $C_i$. \cref{prop:FZaddC}(b), applied in these
			module categories, therefore gives
			$$
			F_i(Z_i)\in\add C_i.
			$$
			Repeating this for $i=t-1,t-2,\ldots,1$ proves the theorem.
		\end{proof}
		
		We finish with the following observation, which is not used later.
		
		\begin{lemma}\label{lem:torsion-classes-under-FZ}
			Let $\T$ be a torsion class in $\mods\Lambda$ satisfying
			$$
			\Gen Z\subseteq\T
			\subseteq{}^{\perp}\tau Z.
			$$
			Then
			$$
			F_Zf_Z(\T)
			=
			f_CF(\T)
			=
			\T\cap\J_{\Lambda/I}(C),
			$$
			and this is a torsion class in $\J_{\Lambda/I}(C)$.
		\end{lemma}
		
		\begin{proof}
			By \cite[Theorem~3.23]{jasso},
			$$
			f_Z(\T)=\T\cap Z^{\perp}
			$$
			is a torsion class in $\J(Z)$. By
			Proposition~\ref{prop:FZ-defined}, the functors $F_Zf_Z$ and $f_CF$ are
			naturally isomorphic on ${}^{\perp}\tau Z$, and hence
			$$
			F_Zf_Z(\T)=f_CF(\T).
			$$
			
			By Lemma~2.1(b), we have
			$
			F(\T)=\T\cap\mods\Lambda/I.
			$
			The torsion-free functor $f_C$ sends a torsion class $\U$ containing
			$\Gen_{\Lambda/I}C$ to its intersection with $C^{\perp_{\Lambda/I}}$.
			Indeed, for $X\in \U$, the canonical sequence
			$$
			0\longrightarrow t_C(X)\longrightarrow X\longrightarrow f_C(X)
			\longrightarrow0
			$$
			has $t_C(X)\in\Gen_{\Lambda/I}C\subseteq \U$ and since
			$\U$ is closed under factor modules, we have $f_C(X)\in \U$.
			Conversely, if $Y\in \U\cap C^{\perp_{\Lambda/I}}$, then
			$f_C(Y)=Y$. Since $C$ is a quotient of $Z$ and $\T$ is closed under quotients, $\Gen_{\Lambda/I}C\subseteq \T\cap \mods \Lambda/I$. Therefore
			$$
			f_CF(\T)
			=
			F(\T)\cap C^{\perp_{\Lambda/I}}.
			$$
			Since $\T\subseteq{}^{\perp}\tau Z$, Lemma~4.7
			gives
			$$
			F(\T)\subseteq
			{}^{\perp_{\Lambda/I}}\tau_{\Lambda/I}C.
			$$
			Consequently, we have
			$$
			f_CF(\T)
			=
			\T\cap C^{\perp_{\Lambda/I}}
			\cap{}^{\perp_{\Lambda/I}}\tau_{\Lambda/I}C
			=
			\T\cap\J_{\Lambda/I}(C).
			$$
			Under the quotient
			presentation supplied by Theorem~\ref{thm:reduction-functor}, identify 
			$\J(Z) \simeq \mods \Gamma$ and $\J_{\Lambda/I}(C) \simeq \mods (\Gamma/I_Z)$.
			The algebra $\Gamma$ is $\tau$-tilting finite by Jasso's reduction theorem
			\cite[Thm. 3.15]{jasso}. Hence Lemma \ref{lem:Gen-preserved}(b), applied to the torsion class $f_Z(\T)$ shows that $F_Zf_Z(\T)$ is a torsion class in $\J_{\Lambda/I}(C)$.
		\end{proof}

		\section{Alternative approaches to lifting $\tau$-rigid modules}
		Note that the natural inclusion functor $\mods \Lambda/I\rightarrow \mods \Lambda$ is right adjoint to $F$. Hence, for $X\in \mods \Lambda$ and $Y\in \mods \Lambda/I$, we have:
		\begin{equation}\label{eq:adj}
			\Hom_{\Lambda/I}(F(X),Y) \simeq \Hom_\Lambda(X,Y).
		\end{equation}  
		This gives the following:
		\begin{lemma}
			\label{lem:adjunction}
			For any full subcategory $\SC$ of $\mods\Lambda/I$, we have
			${}^{\perp}\SC=\varphi^+({}^{\perp_{\Lambda/I}}\SC)$.
		\end{lemma}
		\begin{proof}
			Let $X$ be a $\Lambda$-module.
			We have $X\in {}^{\perp}\SC$ if and only if $\Hom(X,S)=0$ for all $S$ in $\SC$, if and only if $\Hom_{\Lambda/I}(F(X),S)=0$ for all $S\in \SC$, if and only if
			$X\in \varphi^+({}^{\perp_{\Lambda/I}}\SC)$.
		\end{proof}
		
		\begin{corollary}
			\label{cor:liftingtorsion}
			Let $\T$ be a functorially finite torsion class in $\mods \Lambda/I$. 
			Then
			$$\varphi^+(\T)={}^{\perp}\J_{\Lambda/I}(\Pmods(\T))={}^{\perp}(\T^{\perp_{\Lambda/I}}).$$
		\end{corollary}
		\begin{proof}
			Note that $\T={}^{\perp_{\Lambda/I}}\J_{\Lambda/I}(\Pmods(\T))$ by \cite[Thm.\ 2.14]{bhm}.
			Hence, by Lemma~\ref{lem:adjunction} with $\SC=\J_{\Lambda/I}(\Pmods(\T))$, we have
			$$\varphi^+(\T)=\varphi^+({}^{\perp_{\Lambda/I}}\J_{\Lambda/I}(\Pmods(\T)))=
			{}^{\perp}\J_{\Lambda/I}(\Pmods(\T)).$$
			Note also that $\T={}^{\perp_{\Lambda/I}}(\T^{\perp_{\Lambda/I}})$ (see~\cite[IV.1.2]{ass06}).
			So, taking $\SC=\T^{\perp_{\Lambda/I}}$ in Lemma~\ref{lem:adjunction} gives
			$$\varphi^+(\T)=\varphi^+({}^{\perp_{\Lambda/I}}(\T^{\perp_{\Lambda/I}}))=
			{}^{\perp}(\T^{\perp_{\Lambda/I}}).$$
		\end{proof}
		
		\begin{corollary}
			\label{cor:alternativeZ}
			In Setup~\ref{setupCZ}, we have $Z=\Pmods(\varphi^+(\Gen_{\Lambda/I}C))$.
		\end{corollary}
		\begin{proof}
			Apply Corollary~\ref{cor:liftingtorsion} to $\T=\Gen_{\Lambda/I}C$ to obtain
			$$Z=\Pmods({}^{\perp}\J_{\Lambda/I}(C))=\Pmods({}^{\perp}\J_{\Lambda/I}(\Pmods(\Gen_{\Lambda/I}C)))=\Pmods(\varphi^+(\Gen_{\Lambda/I}C)).$$
		\end{proof}
		
		We also note the following lattice-theoretic approach to establishing the indecomposability of $Z$ (Lemma \ref{lem:CZinitial}(c)) and Lemma~\ref{lem:Fperptau}, using the version of $Z$ given by Corollary~\ref{cor:alternativeZ}. We refer to~\cite[Sect.\ 2]{dirrt} for background in lattice theory. We first give a more general result:
		
		\begin{lemma}
			\label{lem:Fperptaunew}
			Let $Z$ be a $\tau$-rigid gen-minimal $\Lambda$-module, and let
			$C$ be a $\tau_{\Lambda/I}$-rigid gen-minimal $\Lambda/I$-module. Suppose that $\varphi^+(\Gen_{\Lambda/I}C)=\Gen Z$. Then
			\begin{itemize}
				\item[(a)] The modules $C$ and $Z$ have the same number of nonisomorphic direct summands. In particular, if $C$ is indecomposable, so is $Z$, and vice versa.
				\item[(b)] We have $F({}^{\perp}\tau Z)={}^{\perp_{\Lambda/I}}\tau_{\Lambda/I}C$.
			\end{itemize}
		\end{lemma}
		\begin{proof}
			By~\cite[Prop.\ 5.7(d)]{dirrt}, the functor $F$ induces a surjective homomorphism of lattices from
			$\tors(\Lambda)$ to $\tors(\Lambda/I)$.
			Note that we may identify $\tors(\Lambda/I)$ with $\tors(\Lambda)/\Theta$, where $\Theta$ is the congruence on $\tors(\Lambda)$ induced by $F$.
			Then $\varphi^+$ coincides with the lattice-theoretic map $\pi^{\uparrow}$ by~\cite[Prop.\ 5.7]{dirrt}, where $\pi^{\uparrow}:\tors(\Lambda)\rightarrow \tors(\Lambda)$
			takes $\T\in \tors(\Lambda)$ to the unique maximum element in its congruence class.
			By the dual of~\cite[Prop.\ 2.2]{reading}, 
			noting~\cite[Prop.\ 5.7]{dirrt}, the functor $F$ induces a one-to-one 
			correspondence between the covers $\T_1,\T_2,\ldots, \T_r$ of 
			$\Gen Z$ in $\tors (\Lambda)$ and the covers of $F(\Gen 
			Z)=\Gen_{\Lambda/I}C$ in $\tors (\Lambda/I)$.
			
			Proof of part (a): By~\cite[Sect.\ 2.3]{air} (see also~\cite[Thm.\ 1.6]{bhm}), we have $Z=\Pmods(\Gen Z)$ and $C=\Pmods(\Gen_{\Lambda/I}C)$. Furthermore, the number of neighbours of a torsion class $\T$ in the Hasse diagram of $\tors(\Lambda)$ (i.e. the number of torsion classes covered by or covering $\T$) is equal to the rank of $\Lambda$ by~\cite[Thms. 2.18, 2.33]{air}. Since the ranks of $\Lambda$ and $\Lambda/I$ are the same (as $I$ is contained in the radical of $\Lambda$), it follows from the above that the number of torsion classes covered by $\Gen Z$ coincides with the number of torsion classes covered by
			$\Gen_{\Lambda/I}C$.
			Observe that if $T$ is a support $\tau$-tilting module with an indecomposable direct summand $X$, then $\Gen T\not=\Gen(T/X)$ if and only if $X$ is split projective in $\Gen T$. 
			Thus the number of torsion classes covered by a fixed torsion class is equal to the number of indecomposable split projective objects in it by~\cite[Thms.\ 2.18, 2.33]{air}. Part (a) now follows.
			
			Proof of part (b): By~\cite[Thm.\ 8.4]{bh2}, we have ${}^{\perp} \tau Z=\bigvee_{i=1}^r
			\T_i$ and, similarly, we have ${}^{\perp_{\Lambda/I}}\tau_{\Lambda/I}C=\bigvee_{i=1}^r F(\T_i)$.
			We thus obtain
			$$F({}^{\perp}\tau Z)=F\left(\bigvee_{i=1}^r \T_i\right)=
			\bigvee_{i=1}^r F(\T_i)={}^{\perp_{\Lambda/I}}\tau_{\Lambda/I} C,$$
			giving (b).
		\end{proof}

		This has the following corollary:
		
		\begin{corollary}\label{cor:F-properties}
			In Setup~\ref{setupCZ}, we have:
			\begin{itemize}
				\item[(a)] $Z$ is an indecomposable $\tau$-rigid $\Lambda$-module;
				\item[(b)] $F(\lperp{\tau Z}) = {}^{\perp_{\Lambda/I}}\tau_{\Lambda/I} C$.
			\end{itemize}
		\end{corollary}
		\begin{proof}
			Let $C$ and $Z$ be as in Setup~\ref{setupCZ}. By Corollary~\ref{cor:alternativeZ},
			we have that $Z=\Pmods(\varphi^+(\Gen_{\Lambda/I}C))$.
			Since $C$ is an indecomposable $\Lambda/I$-module, $Z$ is an
			indecomposable $\Lambda$-module by
			Lemma~\ref{lem:Fperptaunew}(a).
			Part (b) then follows from Lemma~\ref{lem:Fperptaunew}(b).
		\end{proof}
		
		We then have the following. The proof is the same as that of \cref{prop:FZaddC} except that Corollary~\ref{cor:F-properties} is used in the first part for the fact that $F({}^{\perp} \tau Z)={}^{\perp_{\Lambda/I}}\tau_{\Lambda/I}C$.

		\begin{lemma}
			\label{lem:FZaddCnew}
			Let $C$ be a $\tau_{\Lambda/I}$-rigid gen-minimal $\Lambda/I$-module
			and let $Z=\Pmods(\varphi^+(\Gen_{\Lambda/I} C))$.
			Then
			\begin{itemize}
				\item[(a)] $F(Z)\in\add C$ and
				\item[(b)]$\add F(\Pmod(\lperp{\tau Z})) = \add 
				\Pmod(^{\perp_{\Lambda/I}}{\tau_{\Lambda/I} C})$.
			\end{itemize}
		\end{lemma}

		\section{Conjectures}

		\subsection{TF-ordered $\tau$-rigid modules}\label{sec:conj1}
		The concept of TF-ordered $\tau$-rigid modules was introduced by Mendoza and Treffinger
		\cite{mt}. 
		A sequence $(M_1,M_2,\ldots ,M_t)$ of indecomposable $\Lambda$-modules is a \emph{TF-ordered} $\tau$-rigid $\Lambda$-module if $\amalg_{i=1}^t M_i$ is $\tau$-rigid and $M_i\not\in \Gen \amalg_{j=i+1}^t M_j$ for all $i$. 
		Applying \cite[Thm. 5.4]{bm}, Mendoza and Treffinger obtained the following result:
		
		\begin{theorem}
			\cite[Thm.\ 5.1]{mt} \label{thm:TF-bij}
			Let $(M_1,M_2,\ldots ,M_t)$ be a TF-ordered $\tau$-rigid $\Lambda$-module. Then
			$$\omega(M_1,M_2,\ldots ,M_t):=(f_{M_2\amalg \cdots M_t}(M_1),\ldots ,f_{M_t}(M_{t-1}),M_t)$$
			is a $\tau$-exceptional sequence. Furthermore, the map $\omega$ is a bijection from the set of $TF$-ordered $\tau$-rigid $\Lambda$-modules of length $t$ to the set of $\tau$-exceptional sequences of length $t$.
		\end{theorem}
		
		We thus get an induced map $L_I$ from the set of TF-ordered $\tau$-rigid $\Lambda/I$-modules to the set of TF-ordered $\tau$-rigid $\Lambda$-modules given by $L_I(M_1,M_2,\ldots ,M_t)=\omega^{-1}\Delta_I\omega_I(M_1,M_2,\ldots ,M_t)$,
		where $\omega_I$ is the map $\omega$ from \cref{thm:TF-bij} for the algebra $\Lambda/I$.
		
		\begin{conjecture}
			\label{conj:TFordered}
			Let $(M_1,M_2,\ldots ,M_t)$ be a TF-ordered $\tau$-rigid $\Lambda/I$-module and let 
			$$L_I(M_1,M_2,\ldots ,M_t)=(N_1,N_2,\ldots ,N_t).$$
			Then
			$F(N_i)=M_i^{r_i}\amalg X$ where $X\in \add(\amalg_{j=i+1}^t M_j)$ and $r_i\geq 1$ is an integer.
		\end{conjecture}
		
		$$
		\begin{tikzcd}
			\left\{\text{\begin{minipage}{4.2cm}{\begin{center}TF-ordered $\tau$-rigid $\Lambda$-modules of length $t$\end{center}}\end{minipage}}\right\} \arrow[r,"\omega"] & \left\{\text{\begin{minipage}{4.2cm}{\begin{center} $\tau$-exceptional sequences of $\Lambda$-modules of length $t$\end{center}}\end{minipage}}\right\} \\
			\left\{\text{\begin{minipage}{4.2cm}{\begin{center}TF-ordered $\tau$-rigid $\Lambda/I$-modules of length $t$\end{center}}\end{minipage}}\right\} \arrow[r,"\omega_I"] \arrow[u,"L_I",dashed] & \left\{\text{\begin{minipage}{4.5cm}{\begin{center} $\tau$-exceptional sequences of $\Lambda/I$-modules of length $t$\end{center}}\end{minipage}}\right\} \arrow[u,"\Delta_I"] \\
		\end{tikzcd}
		$$
		
		\subsection{Representations of quivers over local algebras}
		
		Let $k$ be an algebraically closed field and $R$ a 
		finite-dimensional commutative local $k$-algebra and let $Q$ be a quiver without oriented cycles. Let $\Lambda$ be the $k$-algebra $RQ\cong R\otimes_k kQ$. We
		recall the following result of Nonis:
		
		\begin{theorem}
			\label{thm:nonis}
			\cite[Thm.\ 5.9]{nonis25}
			The induction functor $\Lambda\otimes_{kQ} -$ gives a bijection $\sigma$ from the set of $\tau$-exceptional sequences of $kQ$-modules of length $t$ to the set of $\tau$-exceptional sequences of $\Lambda$-modules of length $t$.
		\end{theorem}
		
		Let $m$ be the unique maximal ideal of $R$. Then $R/m\cong k$, since $k$ is algebraically closed, see \cite[Cor. 7.10]{am}. Furthermore, the natural map $\pi:R\rightarrow k$ with kernel $m$ induces a surjective algebra morphism $g:RQ\rightarrow kQ$ with kernel $mQ$. Since $R$ is finite-dimensional and local, we have that $m$ is a nilpotent ideal of $R$. Thus $mQ$ is nilpotent in $RQ$, and 
		hence contained in the radical of $RQ$ (by~\cite[I.1.4]{ass06}(c), for example). Thus we can consider the functor $F:\mods \Lambda\rightarrow \mods (\Lambda/mQ) \cong \mods kQ$ given by $F=(\Lambda/mQ) \otimes_{\Lambda} -$, as in the setup here. It is thus natural to make the following conjecture.
		
		\begin{conjecture}
			\label{conj:nonisbijection}
			The injective map $\Delta_{mQ}$ from Theorem~\ref{thm:main1} coincides with the bijection $\sigma$ from Theorem~\ref{thm:nonis}.
		\end{conjecture}
		
		See Example \ref{ex2} for a related example.
		
		\subsection{Nakayama algebras}
		
		The $\tau$-exceptional sequences for Nakayama algebras have been studied in \cite{msa,bkt,kt}.
		We refer to \cite[Sect. V]{ass06} for definitions and details on this class of algebras. 
		
		Let $C_n$ be the Nakayama algebra given by the quotient of the path algebra of an oriented $n$-cycle by the ideal generated by paths of length $n$. We recall the following result of Kaipel and Terland:
		
		\begin{theorem}\cite[Thm.\ 1.2]{kt}
			\label{thm:kt}
			Let $\Lambda$ be a basic connected Nakayama algebra of rank $n$.
			Suppose that the length of every indecomposable projective $\Lambda$-module is at least $n$. Then there is a bijection $\alpha$ from the $\tau$-exceptional sequences over $C_n$ to the $\tau$-exceptional sequences over $\Lambda$.
		\end{theorem}
		
		Let $\Lambda$ be a Nakayama algebra as in Theorem~\ref{thm:kt}.
		Then the quiver of $\Lambda$ is either linearly oriented of Dynkin type $A_n$,
		or an oriented $n$-cycle.  Since the length of every indecomposable projective  module is at least $n$, the cyclic case must occur, except when $n=1$.
		Since $\Lambda$ is finite-dimensional, there must be a relation starting at each vertex, and we may take the shortest vanishing paths at each vertex as defining relations (note that this is not necessarily a minimal set of relations). These must be paths of length at least $n$. It follows that $C_n$ is the quotient of $\Lambda$ by the ideal $I$ generated by the paths of length $n$. This ideal is contained in the radical of $\Lambda$ (which is generated by the arrows). We make the following conjecture:
		
		\begin{conjecture}
			\label{conj:ktbijection}
			Let $\Lambda$ be a basic, connected Nakayama algebra of rank $n$.
			Suppose that the length of every indecomposable projective $\Lambda$-module is at least $n$. Let $I$ be the ideal generated by the paths of length $n$, so that $\Lambda/I\cong C_n$.
			Then the injective map $\Delta_I$ from Theorem~\ref{thm:main1} coincides with the bijection in Theorem~\ref{thm:kt}.
		\end{conjecture}
		
		See Example \ref{ex3} for a special case supporting this conjecture. 
		
		\section{Examples}
		
		\subsection{Example 1}
		
		Let $\Lambda = kQ$ where $Q$ is the quiver
		
		\begin{center}
			\begin{tikzcd}
				1 \arrow[r, "\alpha"] & 2 \arrow[r, "\beta"] & 3
			\end{tikzcd}
		\end{center}
		
		and let $\Gamma  = \Lambda/I$, where $I$ is the 
		ideal generated by the composition $\beta \alpha$.
		
		\newcommand{\colwidth}{3.1cm}
		
		\newcommand{\sing}[1]{%
			$\vcenter{\hbox{$\scriptstyle #1$}}$%
		}
		
		\newcommand{\pair}[2]{%
			$\vcenter{\hbox{%
					\begin{tabular}{@{}c@{}}
						$\scriptstyle #1$\\[-4pt]
						$\scriptstyle #2$
					\end{tabular}%
			}}$%
		}
		
		\newcommand{\triple}[3]{%
			$\vcenter{\hbox{%
					\begin{tabular}{@{}c@{}}
						$\scriptscriptstyle #1$\\[-7pt]
						$\scriptscriptstyle #2$\\[-7pt]
						$\scriptscriptstyle #3$
					\end{tabular}%
			}}$%
		}
		
		\newcommand{\sep}{\hspace{3mm}}
		
		\begin{figure}
			\centering
			
			\begin{tabular}{|
					>{\centering\arraybackslash}m{\colwidth}|
					>{\centering\arraybackslash}m{\colwidth}|
					>{\centering\arraybackslash}m{\colwidth}|
					>{\centering\arraybackslash}m{\colwidth}|}
				\hline
				\small{TF-ordered $\tau$-tilting modules for $\Gamma$} & 
				\small{$\tau$-exceptional sequences for $\Gamma$} & 
				\small{$\tau$-exceptional sequences for $\Lambda$} &
				\small{TF-ordered $\tau$-tilting modules for $\Lambda$}
				\\
				\hline
				\sing{3}\sep\pair{2}{3}\sep\pair{1}{2} &
				\sing{3}\sep\pair{2}{3}\sep\pair{1}{2} &
				\sing{3}\sep\pair{2}{3}\sep\triple{1}{2}{3} &
				\sing{3}\sep\pair{2}{3}\sep\triple{1}{2}{3} \\
				\hline
				\sing{3}\sep\pair{1}{2}\sep\pair{2}{3} &
				\sing{3}\sep\sing{1}\sep\pair{2}{3} &
				\sing{3}\sep\sing{1}\sep\pair{2}{3} &
				\sing{3}\sep\triple{1}{2}{3}\sep\pair{2}{3} \\
				\hline
				\pair{2}{3}\sep\sing{3}\sep\pair{1}{2} &
				\sing{2}\sep\sing{3}\sep\pair{1}{2} &
				\sing{2}\sep\sing{3}\sep\triple{1}{2}{3} &
				\pair{2}{3}\sep\sing{3}\sep\triple{1}{2}{3} \\
				\hline
				\pair{2}{3}\sep\pair{1}{2}\sep\sing{3} &
				\sing{2}\sep\pair{1}{2}\sep\sing{3} &
				\sing{2}\sep\pair{1}{2}\sep\sing{3} &
				\pair{2}{3}\sep\triple{1}{2}{3}\sep\sing{3} \\
				\hline
				\pair{1}{2}\sep\sing{3}\sep\pair{2}{3} &
				\sing{1}\sep\sing{3}\sep\pair{2}{3} &
				\sing{1}\sep\sing{3}\sep\pair{2}{3} &
				\triple{1}{2}{3}\sep\sing{3}\sep\pair{2}{3} \\
				\hline
				\pair{1}{2}\sep\pair{2}{3}\sep\sing{3} &
				\sing{1}\sep\sing{2}\sep\sing{3} &
				\sing{1}\sep\sing{2}\sep\sing{3} &
				\triple{1}{2}{3}\sep\pair{2}{3}\sep\sing{3} \\
				\hline
				\pair{2}{3}\sep\sing{3}\sep\pair{1}{2} &
				\sing{2}\sep\sing{3}\sep\pair{1}{2} &
				\sing{2}\sep\sing{3}\sep\triple{1}{2}{3} &
				\pair{2}{3}\sep\sing{3}\sep\triple{1}{2}{3} \\
				\hline
				\pair{2}{3}\sep\pair{1}{2}\sep\sing{2} &
				\pair{2}{3}\sep\pair{1}{2}\sep\sing{2} &
				\pair{2}{3}\sep\triple{1}{2}{3}\sep\sing{2} &
				\pair{2}{3}\sep\triple{1}{2}{3}\sep\sing{2} \\
				\hline
				\pair{1}{2}\sep\pair{2}{3}\sep\sing{2} &
				\sing{1}\sep\pair{2}{3}\sep\sing{2} &
				\sing{1}\sep\pair{2}{3}\sep\sing{2} &
				\triple{1}{2}{3}\sep\pair{2}{3}\sep\sing{2} \\
				\hline
				\sing{3}\sep\pair{1}{2}\sep\sing{1} &
				\sing{3}\sep\pair{1}{2}\sep\sing{1} &
				\sing{3}\sep\triple{1}{2}{3}\sep\sing{1} &
				\sing{3}\sep\triple{1}{2}{3}\sep\sing{1} \\
				\hline
				\pair{1}{2}\sep\sing{3}\sep\sing{1} &
				\pair{1}{2}\sep\sing{3}\sep\sing{1} &
				\pair{1}{2}\sep\sing{3}\sep\sing{1} &
				\triple{1}{2}{3}\sep\sing{3}\sep\sing{1} \\
				\hline
				\pair{1}{2}\sep\sing{1}\sep\sing{3} &
				\pair{1}{2}\sep\sing{1}\sep\sing{3} &
				\pair{1}{2}\sep\sing{1}\sep\sing{3} &
				\triple{1}{2}{3}\sep\sing{1}\sep\sing{3} \\
				\hline
				&
				&
				\triple{1}{2}{3}\sep\pair{1}{2}\sep\sing{1} &
				\triple{1}{2}{3}\sep\pair{1}{2}\sep\sing{1} \\
				\hline
				&
				&
				\sing{2}\sep\triple{1}{2}{3}\sep\pair{1}{2} &
				\sing{2}\sep\triple{1}{2}{3}\sep\pair{1}{2} \\
				\hline
				&
				&
				\triple{1}{2}{3}\sep\sing{2}\sep\pair{1}{2} &
				\triple{1}{2}{3}\sep\sing{2}\sep\pair{1}{2} \\
				\hline
				&
				&
				\triple{1}{2}{3}\sep\sing{1}\sep\sing{2} &
				\triple{1}{2}{3}\sep\pair{1}{2}\sep\sing{2} \\
				\hline
			\end{tabular}
			
			\caption{For the algebras $\Lambda$ and $\Gamma$ in Example 1, the third column gives the lifts, using $\Delta_I$, of the $\tau$-exceptional sequences in the second column. The corresponding TF-ordered $\tau$-tilting modules are also given.}
			\label{fig:A3,bijections}
		\end{figure}
		
		Both $\Lambda$ and $\Gamma$ are Nakayama algebras. Hence, it is natural to denote the indecomposable objects in their module categories by their composition series. E.g. 
		$X= \text{\triple{1}{2}{3}}$ denotes a module of length $3$, with the simple $S_1$ in the top (i.e. $S_1$ is the unique simple factor module), and the simple $S_3$ in the socle (i.e. $S_3$ is the unique simple submodule).
		
		We note first that $\mods \Lambda$ has six indecomposable objects,
		$$\left\{\text{\triple{1}{2}{3}}, \, \text{\pair{1}{2}}, \,\text{\pair{2}{3}}, \, \text{\sing{1}}, \, \text{\sing{2}}, \, \text{\sing{3}} \right\} $$
		while the indecomposables in $\mods \Gamma$ are the same, except for
		\triple{1}{2}{3}. All indecomposables are $\tau$-rigid in both
		$\mods \Gamma$ and $\mods \Lambda$.
		
		We next note that 
		$\phi^+\left(\Gen \, \text{\pair{1}{2}}\right) = \Gen
		\text{\triple{1}{2}{3}}$. 
		So the lifting of $C= \text{\pair{1}{2}}$
		is $Z= \text{\triple{1}{2}{3}}$.
		For the other four indecomposables in $\mods \Gamma$, the lifting is just the identity.
		
		This implies that the two $\tau$-exceptional sequences in $\mods \Lambda$ ending in \pair{1}{2} do not appear in the image of $\Delta_I$. 
		
		We next note that $\J_{\Gamma}(2)=
		\add\left({\text{\pair{2}{3}} \amalg \text{\sing{1}}}\right)$, which is a semisimple module category of rank $2$, and hence has two 
		$\tau$-exceptional sequences.
		However, the category $\J_{\Lambda}(2)$ has the indecomposables 
		$\left\{\text{\pair{2}{3}}, \, \text{\triple{1}{2}{3}}, \text{\sing{1}}\right\}$ and is equivalent to the module category of an algebra of Dynkin type $A_2$. Note that $A_2$ has three exceptional sequences.
		We have that the
		$\tau$-exceptional sequence
		$\left(\text{\triple{1}{2}{3}},\text{\sing{1}},\text{\sing{2}}\right)$
		is not in the image of $\Delta_I$.
		
		Similarly, we have that $\J_{\Gamma}(1)$ is semisimple, while 
		$\J_{\Lambda}(1)$ is equivalent to 
		a module category of type $A_2$. The $\tau$-exceptional sequence $\left(\text{\triple{1}{2}{3}}, \text{\pair{1}{2}}, \text{\sing{1}}\right)$
		is not in the image of $\Delta_I$.
		
		Together, these observations explain why there are four 
		$\tau$-exceptional sequences in $\mods \Lambda$ that do not occur in the image of $\Delta_I$. See the table in Figure \ref{fig:A3,bijections}, which also gives the explicit connection 
		between TF-ordered $\tau$-tilting modules and $\tau$-exceptional sequences, as described in general in Section \ref{sec:conj1}. 
		
		\subsection{Example 2}\label{ex2}

		Consider the quiver $Q$ given by
		
		$$
		
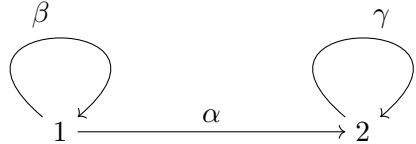
\begin{tikzpicture}
			\node (1) at (0,0) {$1$};
			\node (2) at (4,0) {$2$};
			
			\draw[->] (1) -- node[above] {$\alpha$} (2);
			
			\draw[->] 
			(1) to[out=140,in=40,looseness=12] 
			node[above left] {$\beta$} (1);
			
			\draw[->] 
			(2) to[out=140,in=40,looseness=12] 
			node[above right] {$\gamma$} (2);
		\end{tikzpicture}
		$$
		
		Let $L$ be the ideal generated by the paths $\gamma \alpha \beta, \beta^2, \gamma^2$, and let $L'$ be the ideal generated by the loops 
		$\beta$ and $\gamma$, and let $\Lambda = kQ/L$. 
		
		Furthermore, consider the ideal $I= L'/L$ in $\Lambda$.
		Then $\Lambda/I$ is isomorphic to the hereditary path algebra of
		type $A_2$. 
		
		In this example, one can verify by direct computation that there is a bijection between the indecomposable 
		$\tau$-rigid modules in $\mods \Lambda/I$ and those in $\mods \Lambda$, and the map $\Delta_I$ is also a bijection. 
		Note that this is a string algebra, so the computations of $\tau$-rigid modules can be carried out using \cite{string}. 
		
		Furthermore, let
		\[
		B=R\otimes_k kQ',
		\]
		where $R=k[x]/(x^2)$ and $Q'$ is a quiver of Dynkin type $A_2$.
		Then
		\[
		B\cong kQ/K,
		\qquad
		K=\langle \gamma\alpha-\alpha\beta,\beta^2,\gamma^2\rangle.
		\]
		Since $L\subseteq K$, the algebra $B$ is a factor algebra of $\Lambda$.
		The algebra $B$ is a special case of Theorem~7.3, and hence its
		$\tau$-exceptional sequences are in bijection with those of the
		hereditary path algebra of type $A_2$.

		\subsection{Example 3}\label{ex3}

		Consider the quiver $Q$ given by 
		
		$$
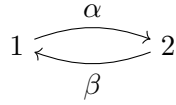
\begin{tikzpicture}
			\node (1) at (0,0) {$1$};
			\node (2) at (2,0) {$2$};
			
			\draw[->] (1) to[bend left=20] node[above] {$\alpha$} (2);
			\draw[->] (2) to[bend left=20] node[below] {$\beta$} (1);
		\end{tikzpicture}$$
		
		Let $J$ be the ideal in $kQ$ generated by the arrows, and consider the ideal $I= J^2/J^3$ in $\Lambda = kQ/J^3$.
		
		Both $\mods \Lambda$ and $\mods \Lambda/I$ have four indecomposable $\tau$-rigid objects, the two simples and the two indecomposable projectives.
		
		Then it is easily verified that $\Delta_I$ is a bijection, and is 
		given by replacing each indecomposable projective object in $\mods \Lambda/I$ with the corresponding projective in $\mods \Lambda$, and similarly for simples.
		
		Note that this example is a special case of a more general 
		construction for Nakayama algebras given by Kaipel and Terland \cite{kt}, and illustrates
		\cref{conj:ktbijection}.
		
		\subsection*{Use of AI}
		The authors used ChatGPT (GPT-6) for language editing and suggestions to improve the presentation of the manuscript.

	\end{document}